\documentclass[english]{article}
\usepackage[T1]{fontenc}
\usepackage[latin9]{inputenc}
\usepackage{color}
\usepackage{amsmath}
\usepackage{amssymb}

\makeatletter
\usepackage{amsthm}\usepackage{fullpage}\usepackage{makeidx}\usepackage{amsfonts}\usepackage{latexsym}
\numberwithin{equation}{section}

\numberwithin{figure}{section}
\def\C{{\mathbb{C}}}
\def\Q{{\mathbb{Q}}}
\def\Z{{\mathbb{Z}}}
\def\R{{\mathbb{R}}}
\def\N{{\mathbb{N}}}

\def\cY{{\cal Y}}

\def \End{{\rm End}}
\def\qdim{{\rm qdim}}

\newcommand{\Lx}[0]{x^{j + 1} \frac{\partial}{\partial x}}
\newcommand{\Lo}[0]{x \frac{\partial}{\partial
x}}

\theoremstyle{definition}
\newtheorem{lemma}{Lemma}[section]
\newtheorem{theorem}[lemma]{Theorem}
\newtheorem{corollary}[lemma]{Corollary}
\newtheorem{proposition}[lemma]{Proposition}
\newtheorem{definition}[lemma]{Definition}

\newtheorem{remark}[lemma]{Remark}\usepackage{times}

\newtheorem*{CPMA}{Theorem A}
\newtheorem*{CPMB}{Theorem B}
\newtheorem*{CPMC}{Theorem C}
\newcommand{\cc}{{\mathcal C}}
\newcommand{\FPdim}{{\rm FPdim}\,}

\usepackage{babel}

\usepackage[blocks]{authblk}
\title{Fusion products of twisted modules in permutation orbifolds}

\author{Chongying Dong}
\affil{Department of Mathematics, University of
California, Santa Cruz, CA 95064 USA}
\author{Haisheng Li}
\affil{ Department of Mathematical Sciences, Rutgers University, Camden, NJ 08102 USA}
\author{Feng Xu}

\affil{Department of Mathematics, University of California, Riverside, CA 92521 USA}
\author{Nina Yu}
\affil{School of Mathematical Sciences, Xiamen University, Fujian 361005 CHINA}

\usepackage{babel}

\usepackage{babel}

\makeatother

\begin{document}

\title{Fusion products of twisted modules in permutation orbifolds}
\maketitle
\begin{abstract}
Abstract: Let $V$ be a vertex operator algebra, $k$ a positive integer
and $\sigma$ a permutation automorphism of the vertex operator algebra
$V^{\otimes k}$. In this paper, we determine the fusion product of
any $V^{\otimes k}$-module with any $\sigma$-twisted $V^{\otimes k}$-module.
\end{abstract}

\section{Introduction}
This paper is about permutation orbifold vertex operator algebras.
The theory of permutation orbifolds studies the representations of the tensor product vertex operator algebra
$V^{\otimes k}$ with the natural action of the symmetric group $S_{k}$ as an automorphism group, where
$V$ is a vertex operator algebra and $k$ is a positive integer.
The study of permutation orbifolds was initiated in \cite{BHS},
where the twisted sectors (modules), genus one characters and their modular
transformations, as well as the fusion rules for cyclic permutations for affine vertex operator algebras and the Virasoro vertex operator algebras
were studied. The genus one characters and modular transformation properties of permutation
orbifolds for a general rational conformal field theory were given in \cite{Ba}.

The study on permutation orbifolds for general vertex operator algebras
was started by Barron, Dong and Mason in \cite{BDM}, where for a general
vertex operator algebra $V$ with any positive integer $k$, an intrinsic connection
between  twisted modules for tensor product vertex operator algebra $V^{\otimes k}$ with respect to permutation automorphisms
and $V$-modules was found.
More specifically, let $\sigma$ be a $k$-cycle which is viewed as an automorphism of $V^{\otimes k}$.
Then, for any $V$-module $(W,Y_{W}(\cdot,z))$, a canonical $\sigma$-twisted $V^{\otimes k}$-module structure
on $W$ was obtained, which was denoted by $T_{\sigma}(W)$. 
Furthermore, it was proved therein that this gives rise to an isomorphism of the categories of weak, admissible
and ordinary $V$-modules and the categories of weak,
admissible and ordinary $\sigma$-twisted $V^{\otimes k}$-modules, respectively.
The $C_{2}$-cofiniteness of permutation orbifolds and general cyclic
orbifolds was established later in \cite{A1,A2,M1,M2}. On the other hand, an equivalence
of two constructions (see \cite{FLM}, \cite{Le}, \cite{BDM}) of twisted modules for permutation
orbifolds of lattice vertex operator algebras was given in \cite{BHL}.
The permutation orbifolds of the lattice vertex operator algebras
with $k=2$ and $k=3$ were extensively studied in \cite{DXY1, DXY2, DXY3}.
In particular, the fusion rules for the permutation orbifolds were
determined. On the other hand, a study on permutation orbifolds in the context of conformal nets was
given in \cite{KLX}, where irreducible representations of the cyclic
orbifolds were determined and fusion rules were given for $k=2$.

Note that intertwining operator and fusion rule for modules were introduced in \cite{FHL}, while
a theory of tensor products for modules for vertex operator algebras was developed and has been extensively studied
by Huang and Lepowsky (see \cite{HL1-2}, \cite{HL3}, \cite{H1}, \cite{H}).
 Intertwining operators among three twisted modules were studied by Xu (see \cite{Xu})
and a notion of tensor product for twisted modules was introduced in \cite{DLM1}.

In this paper, we study fusion products of $V^{\otimes k}$-modules with $\sigma$-twisted
$V^{\otimes k}$-modules for any $\sigma\in S_{k}$. More specifically,
by using the explicit construction of the $\sigma$-twisted $V^{\otimes k}$-modules due to \cite{BDM},
we completely determine the fusion products. To achieve this goal, we study fusion products of twisted modules
and obtain the associativity of the fusion product of twisted modules by using a theorem
of  \cite{KO} and \cite{CKM}. We also generalize and use a result of \cite{Li1} and \cite{Li2} on an analogue of Hom-functor.
More specifically, given any $\sigma$-twisted modules $W_{2}$ and $W_{3}$
for a vertex operator algebra $V$ with a finite order automorphism $\sigma$,
we introduce the space $\mathcal{H}(W_{2},W_{3})$
of generalized intertwining operators from $W_{2}$ to $W_{3}$ and prove that
$\mathcal{H}(W_{2},W_{3})$ has a natural weak $V$-module structure.
Furthermore, we prove that for any $V$-module $W_{1},$ the
space $\text{Hom}_{V}(W_{1},\mathcal{H}(W_{2},W_{3}))$
is naturally isomorphic to the space of intertwining operators of
type $\binom{W_3}{W_{1}\ W_{2}}$. As one of the main results,
we show that for any $V$-modules $M,N,W$, there is a canonical linear isomorphism
from the space $I_{V}\binom{W}{M\ N}$ of intertwining operators of the indicated type to the space $I_{V^{\otimes k}}\binom{T_{\sigma}(W)}{M^1\ T_{\sigma}(N)}$ of intertwining operators, where
$M^1=M\otimes V^{\otimes (k-1)}$ is viewed as a $V^{\otimes k}$-module.

Now, we describe the contents of this paper with more details.
Let $V$ be any vertex operator algebra and let $k$ be a positive integer.
Recall from \cite{BDM} that for $\sigma=(12\cdots k)\in \text{Aut}(V^{\otimes k})$,
every $\sigma$-twisted $V^{\otimes k}$-module is isomorphic to
$T_{\sigma}(W)$ for some $V$-module $W$ and $T_{\sigma}(W)$ is irreducible if and only if $W$ is irreducible.
On the other hand,  from \cite{FHL},  irreducible $V^{\otimes k}$-modules are
classified as $M_1\otimes \cdots \otimes M_k$, where $M_1,\dots,M_k$ are irreducible $V$-modules.
Our first goal is to determine the fusion product
$$ (M_1\otimes \cdots \otimes M_k)\boxtimes T_{\sigma}(W).$$
We first consider the special case where $M_j=V$ for $2\le j\le k$.
While it is natural to start with a special case, actually the main motivation comes from the following observation
\begin{align}\label{product=iterate}
\left(M_{1}\otimes\cdots\otimes M_{k}\right)\simeq
\boxtimes_{i=1}^{k}(V^{\otimes(i-1)}\otimes M_{i}\otimes V^{\otimes(k-i)}).
\end{align}

The following is our first main result:

\begin{CPMA}\label{t:A}
Let $V$ be any vertex operator algebra and set $\sigma=(12\cdots k)\in \text{Aut}(V^{\otimes k})$
with $k$ a positive integer.
Let $M$ and $N$ be any $V$-modules such that
a tensor product $M\boxtimes N$ exists. Then
$$\left(M\otimes V^{\otimes (k-1)}\right)\boxtimes_{V^{\otimes k}}T_{\sigma}(N)
\simeq T_{\sigma}\left(M\boxtimes N\right).$$
\end{CPMA}

To prove this theorem, we generalize and use a result of \cite{Li1} and \cite{Li2}
on an analogue of the classical Hom-functor.
Let $V$ be any vertex operator algebra with a finite order automorphism $\tau$.
Let $W_{2}$ and $W_{3}$ be $\tau$-twisted $V$-modules.
We first introduce the space $\mathcal{H}(W_{2},W_{3})$
of what are called generalized intertwining operators from $W_{2}$ to $W_{3}$,
where the definition  of a generalized intertwining operator $\phi(x)$ from $W_{2}$ to $W_{3}$
captures the essential properties of $I(w,x)$ for an intertwining operator $I(\cdot,x)$ of type
$\binom{W_3}{W\ W_2}$ with $W$ a $V$-module and with $w\in W$.
More specifically, in addition to the lower truncation condition, $\phi(x)$ satisfies the conditions that
$[L(-1),\phi(x)]=\frac{d}{dx}\phi(x)$ and that
for every $v\in V$, there exists a nonnegative integer $n$ such that
$$(x_1-x)^nY_{W_3}(v,x_1)\phi(x)=(x_1-x)^n\phi(x)Y_{W_2}(v,x_1).$$
We prove that
$\mathcal{H}(W_{2},W_{3})$ has a natural weak $V$-module structure and
that for any $V$-module $W_{1},$ the
linear space $\text{Hom}_{V}(W_{1},\mathcal{H}(W_{2},W_{3}))$
is naturally isomorphic to the space of intertwining operators of
type $\binom{W_3}{W_{1}\ W_{2}}$.
Coming back to vertex operator algebra $V^{\otimes k}$, let $M,N,W$ be $V$-modules.
By exploiting certain techniques from \cite{BDM}  we obtain a linear isomorphism between
$I_{V^{\otimes k}}\binom{T_{\sigma}(W)}{M^{1}\ T_{\sigma}(N)}$
and $I_{V}\binom{W}{M\ N}$,  where
$M^1=M\otimes V^{\otimes (k-1)}$ . Furthermore, we use this natural isomorphism
to obtain the fusion product isomorphism relation  of Theorem A.

As it was indicated before, the main idea is to prove Theorem A first and then use the relation (\ref{product=iterate}) and
 the associativity of the tensor product to deal with the general case.
Using Theorem A and a general result (Lemma \ref{sigma-tau}), we show
$$\left(V^{\otimes (r-1)}\otimes M\otimes V^{\otimes (k-r)}\right)\boxtimes_{V^{\otimes k}}T_{\sigma}(N)
\simeq \left(M\otimes V^{\otimes (k-1)}\right)\boxtimes_{V^{\otimes k}}T_{\sigma}(N)
\simeq T_{\sigma}\left(M\boxtimes N\right)$$
for $2\le r\le k$.

On the other hand, to carry out this plan we shall need the associativity of the tensor product for twisted modules.
Let $V$ be a regular and self-dual vertex operator algebra of CFT type and
let $G$ be a finite abelian automorphism group of $V.$ For any
$g$-twisted $V$-module $M$ and $h$-twisted $V$-module $N$ with $g,h\in G$, we establish that
a tensor product of $M$ and $N$ exists and the tensor product functor is associative.

This result is achieved by using a theorem of \cite{KO} and \cite{CKM} and a theorem of \cite{H}, where
our approach is essentially the same as that in \cite{CKM}.
First, it was proved in \cite{CM,M2} that $V^{G}$ is also a regular and self-dual vertex operator algebra
of CFT type. Then by  \cite{H}, the $V^{G}$-module category ${\cal C}_{V^{G}}$
is a modular tensor category. On the other hand,  a fusion category ${\rm Rep}(V)$ was introduced in
\cite{KO} and \cite{CKM} as a subcategory of ${\cal C}_{V^{G}}$.
This in particular implies that the tensor product functor for ${\rm Rep}(V)$ is associative.
It turns out that for any finite order automorphism
$g$ of $V,$ a $g$-twisted $V$-module is an object in ${\rm Rep}(V).$
Then it is shown that the tensor product of two twisted modules in the sense of
\cite{Xu} and \cite{DLM1} exists and is equivalent to the tensor product in the category
${\rm Rep}(V).$ Consequently, we obtain the existence and the associativity of the tensor product
in the sense of  \cite{Xu} and \cite{DLM1}. One can easily check Theorem A holds for intertwining operators with
logarithm functions with the same proof.

The following is the second main theorem of this paper:

\begin{CPMB}\label{t:B}
Let $V$ be a regular and self-dual vertex operator algebra of CFT type, and  let $M_{1},...,M_{k},N$ be $V$-modules.
Let $\sigma$ be a $k$-cycle permutation.   Then
\[
\left(M_{1}\otimes\cdots\otimes M_{k}\right)\boxtimes_{V^{\otimes k}}T_{\sigma}\left(N\right)\simeq T_{\sigma}\left(M_{1}\boxtimes_{V}\cdots\boxtimes_{V}M_{k}\boxtimes_{V}N\right).
\]
\end{CPMB}

Now, let $\sigma\in S_k$ be an arbitrary permutation. A classical fact is that the conjugacy class of $\sigma$ is uniquely
determined by a partition $\kappa$ of $k$, i.e.,
$\kappa=(k_1,\dots,k_s)$ is a sequence of positive integers such that
$$k_1\ge k_2\ge\cdots \ge k_s,\   \   \   \   k_1+k_2+\cdots +k_s=k.$$
Define $\sigma_{\kappa}=\sigma_1\cdots \sigma_s$, where $\sigma_1=(12\cdots k_1)$, $\sigma_2=((k_1+1)\cdots (k_1+k_2))$,  and so on. Then $\sigma_{\kappa}=\mu\sigma\mu^{-1}$ for some permutation $\mu$.
It is known that the categories of $\sigma$-twisted modules and $\sigma_{\kappa}$-twisted modules are isomorphic canonically. Using \cite{BDM} and \cite{FHL}, we show that
irreducible $\sigma_{\kappa}$-twisted $V^{\otimes k}$-modules are classified as
\begin{align}
T_{\sigma_{\kappa}}(N_1,\dots,N_s):= T_{\sigma_1}(N_1)\otimes \cdots \otimes T_{\sigma_s}(N_s),
\end{align}
where $N_1,\dots,N_s$ are irreducible $V$-modules. Furthermore,
irreducible $\sigma$-twisted $V^{\otimes k}$-modules are classified as
\begin{align}
T_{\sigma}(N_1,\dots,N_s):=  \left(T_{\sigma_{\kappa}}(N_1,\dots,N_s)\right)^{\mu},
\end{align}
where for any $\sigma_{\kappa}$-twisted $V^{\otimes k}$-module $(W,Y_W)$,
$(W^\mu,Y_{W}^{\tau})$ is a $\sigma$-twisted $V^{\otimes k}$-module with $W^\mu=W$ as a vector space and
$$Y_{W}^{\mu}(a,x)=Y_{W}(\mu (a),x)\   \   \   \mbox{ for }a\in V^{\otimes k}.$$

The following is the third main result of this paper:

\begin{CPMC}\label{t:C}
 Assume that $V$ is a regular and self-dual vertex operator algebra of CFT type. Let $k$ be a positive integer and $\sigma\in S_k\subset \text{Aut}(V^{\otimes k})$.
Suppose
$\mu \sigma\mu^{-1}=\sigma_{\kappa}$, where
$\kappa=(k_1,\dots,k_s)$ is a partition of $k$ and $\mu\in S_k$.
Let $M_{1},\dots,M_{k}$ and $N_{1},\dots,N_{s}$ be irreducible $V$-modules. Then
\begin{align}
\left(M_{1}\otimes\cdots\otimes M_{k}\right)\boxtimes T_{\sigma_{\kappa}}(N_{1},N_2,\dots,N_s)\simeq
T_{\sigma_{\kappa}}( M^{[k_1]} \boxtimes N_{1},\dots, M^{[k_s]} \boxtimes N_{s})
\end{align}
as a $\sigma_{\kappa}$-twisted $V^{\otimes k}$-module, and
\begin{align}
\left(M_{1}\otimes\cdots\otimes M_{k}\right)^{\mu}\boxtimes T_{\sigma}(N_{1},N_2,\dots,N_s)
\simeq T_{\sigma}( M^{[k_1]} \boxtimes N_{1},\dots, M^{[k_s]} \boxtimes N_{s})
\end{align}
as a $\sigma$-twisted $V^{\otimes k}$-module,
where $M^{[k_1]}=\boxtimes_{i=1}^{k_1}M_i$, $M^{[k_2]}=\boxtimes_{i=k_1+1}^{k_1+k_2}M_i$, and so on,
are $V$-modules.
\end{CPMC}

We can replace the regularity assumption in Theorems B and C by the associativity assumption of tensor product in $(V^{\otimes k})^{\left\langle \sigma\right\rangle} $-module category, which is what we really need.

This paper is organized as follows. In Section 2, we review some basics
on vertex operator algebras, we also present the analytic function
version of the twisted Jacobi Identity. In Section 3, we prove that
for a regular, self-dual vertex operator algebra $V$ of CFT type
and a finite abelian automorphism group $G$ of $V$, the tensor product
$M\boxtimes_{V}N$ for any $g$-twisted $V$-module $M$ and $h$-twisted
$V$-module $N$ exists and is associative for $g,h\in G.$ In Section 4, we present the weak $V$-module
$\mathcal{H}(M_2,M_3)$ of generalized intertwining operators from a $\sigma$-twisted $V$-module $M_2$ to
another $\sigma$-twisted $V$-module $M_3$.
In Section 5,  for any $V$-modules $M,N,W$, we construct
a linear isomorphism between the spaces  $I_{V}\binom{W}{M\ N}$ and
$I_{V^{\otimes k}}\binom{T_{\sigma}(W)}{M^{1}\ T_{\sigma}(N)}$, and 
prove Theorem A:
$\left(M\otimes V^{\otimes (k-1)}\right)\boxtimes T_{\sigma}(N)=T_{\sigma}\left(M\boxtimes_{V}N\right)$.
Finally, in Section
6, by applying associativity of the tensor product $\boxtimes$ in
Section 3, we obtain Theorems B and C.

{\bf Notations:}  For this paper, we use $\Z_+$ to denote the set of nonnegative integers.

\section{Basics}

Let $\left(V,Y,\mathbf{1},\omega\right)$ be a vertex operator algebra in the sense of
\cite{FLM} and \cite{FHL} (cf. \cite{Bo}). First, we recall the definitions of an automorphism
of $V$ and a $g$-twisted $V$-module for  a finite order automorphism $g$ of $V$ (see \cite{FLM}, \cite{DLM1}).

\begin{definition} An \emph{automorphism} of a vertex operator
algebra $V$ is a linear isomorphism $g$ of $V$ such that $g\left(\omega\right)=\omega$
and $gY\left(v,z\right)g^{-1}=Y\left(gv,z\right)$ for any $v\in V$.
Denote by $\mbox{Aut}\left(V\right)$ the group of all automorphisms
of $V$. \end{definition}

A straightforward consequence is that every automorphism $g$ of $V$ preserves the vacuum vector ${\bf 1}$,
i.e., $g({\bf 1})={\bf 1}$. On the other hand,
a standard fact is that for any subgroup $G\le\mbox{Aut}\left(V\right)$, the set of $G$-fixed points
$$V^{G}:=\left\{ v\in V\ |\ g\left(v\right)=v \  \mbox{ for  }g\in G\right\} $$
is a vertex operator subalgebra.

Let $g$ be a finite order automorphism of $V$ with a period $T$ in the sense that $T$ is a positive integer
such that $g^T=1$. Then
\begin{align}
V=\oplus_{r=0}^{T-1}V^r,
\end{align}
where $V^{r}=\left\{ v\in V\ |\  gv=e^{2\pi i r/T}v\right\} $ for $r\in \Z$.
Note that for $r,s\in \Z$, $V^r=V^s$ if $r\equiv s\ ({\rm mod}\; T)$.

\begin{definition} A \emph{weak $g$-twisted $V$-module} is
a vector space $M$ with a linear map
\begin{align*}
Y_{M}(\cdot,z):\   & V\to\left(\text{End}M\right)[[z^{1/T},z^{-1/T}]]\\
 & v\mapsto Y_{M}\left(v,z\right)=\sum_{n\in \frac{1}{T}\mathbb{Z}}v_{n}z^{-n-1}\ \left(v_{n}\in\mbox{End}M\right),
\end{align*}
which satisfies the following conditions: For all $u\in V^{r}$,
$v\in V$, $w\in M$ with $0\le r\le T-1$,

\[
Y_{M}(u,z)=\sum_{n\in \frac{r}{T}+\mathbb{Z}}u_{n}z^{-n-1},
\]

\[
u_{n}w=0\  \   \mbox{ for $n$ sufficiently large},
\]

\[
Y_{M}(\mathbf{1},z)=Id_{M},
\]

\[
z_{0}^{-1}\text{\ensuremath{\delta}}\left(\frac{z_{1}-z_{2}}{z_{0}}\right)Y_{M}(u,z_{1})Y_{M}(v,z_{2})
-z_{0}^{-1}\delta\left(\frac{z_{2}-z_{1}}{-z_{0}}\right)Y_{M}(v,z_{2})Y_{M}(u,z_{1})
\]

\begin{equation}
=z_{2}^{-1}\left(\frac{z_{1}-z_{0}}{z_{2}}\right)^{-\frac{r}{T}}\delta\left(\frac{z_{1}-z_{0}}{z_{2}}\right)Y_{M}\left(Y\left(u,z_{0}\right)v,z_{2}\right),\label{Jacobi for twisted V-module}
\end{equation}
where $\delta\left(z\right)=\sum_{n\in\mathbb{Z}}z^{n}$.
\end{definition}

Notice that  for any weak $g$-twisted $V$-module $(M,Y_M)$, from definition we have
\begin{align}
z^{\frac{r}{T}}Y_M(u,z)w\in M((z))\   \   \   \   \mbox{ for }u\in V^r,\ r\in \Z,\ w\in M.
\end{align}
On the other hand, we have (see \cite{DL})
$$z_{2}^{-1}\left(\frac{z_{1}-z_{0}}{z_{2}}\right)^{-\frac{r}{T}}\delta\left(\frac{z_{1}-z_{0}}{z_{2}}\right)
=z_{1}^{-1}\left(\frac{z_{2}+z_{0}}{z_{1}}\right)^{\frac{r}{T}}\delta\left(\frac{z_{1}-z_{0}}{z_{2}}\right).$$

\begin{remark} \label{equiv. of Jacobi Idenitty} It is well known that the twisted Jacobi
identity (\ref{Jacobi for twisted V-module}) is equivalent to the
following associativity relation
\begin{equation}
\left(z_{0}+z_{2}\right)^{l+\frac{r}{T}}Y_{M}\left(u,z_{0}+z_{2}\right)Y_{M}\left(v,z_{2}\right)w=\left(z_{2}+z_{0}\right)^{l+\frac{r}{T}}Y_{M}\left(Y\left(u,z_{0}\right)v,z_{2}\right)w\label{associativity formula}
\end{equation}
for $w\in M$, where $l$ is a nonnegative integer such that $z^{l+\frac{r}{T}}Y_{M}\left(u,z\right)w$
involves only nonnegative integral powers of $z$, and commutator relation
\begin{align}\label{twisted-commutator-formula}
 & \left[Y_{M}\left(u,z_{1}\right),Y_{M}\left(v,z_{2}\right)\right]\nonumber\\
 =\ & \text{Res}_{z_{0}}z_{2}^{-1}\left(\frac{z_{1}-z_{0}}{z_{2}}\right)^{-\frac{r}{T}}\delta\left(\frac{z_{1}-z_{0}}{z_{2}}\right)Y_{M}\left(Y\left(u,z_{0}\right)v,z_{2}\right).
\end{align}

From \cite{DLMi} (cf. Definition 2.7 and Lemma 2.9 of \cite{LTW}), (\ref{associativity formula})
can be equivalently replaced with the property that for $u,v\in V$,
there exists a nonnegative integer $m$ such that

\[
(z_{1}-z_{2})^{m}Y_{M}(u,z_{1})Y_{M}(v,z_{2})\in \text{Hom}\left(M,M((z_{1}^{\frac{1}{T}},z_{2}^{\frac{1}{T}}))\right)
\]
and
\[
z_{0}^{m}Y_{M}\left(Y(u,z_{0})v,z_{2}\right)=(z_{1}-z_{2})^{m}Y_{M}(u,z_{1})Y_{M}(v,z_{2})|_{z_{1}^{\frac{1}{T}}=(z_{2}+z_{0})^{\frac{1}{T}}.}
\]
\end{remark}

\begin{definition}An \emph{admissible $g$-twisted $V$-module}
is a weak $g$-twisted module with a $\frac{1}{T}\mathbb{Z}_{+}$-grading $M=\oplus_{n\in\frac{1}{T}\mathbb{Z}_{+}}M(n)$
such that $u_{m}M\left(n\right)\subset M\left(\mbox{wt}u-m-1+n\right)$
for homogeneous $u\in V$ and $m,n\in\frac{1}{T}\mathbb{Z}.$ $ $
\end{definition}

Note that if $M=\oplus_{n\in\frac{1}{T}\mathbb{Z}_{+}}M(n)$ is an irreducible admissible $g$-twisted $V$-module,
then there is a complex number
$\lambda_{M}$ such that $L(0)|_{M(n)}=\lambda_{M}+n$ for all $n.$
As a convention, we assume $M(0)\ne0$, and $\lambda_{M}$ is called the weight
or conformal weight of $M.$

\begin{definition}A $g$-\emph{twisted $V$-module} is a weak $g$-twisted $V$-module
$M$ which carries a $\mathbb{C}$-grading induced by the spectrum
of $L(0)$ where $L(0)$ is the component operator of $Y(\omega,z)=\sum_{n\in\mathbb{Z}}L(n)z^{-n-2}.$
That is, we have $M=\bigoplus_{\lambda\in\mathbb{C}}M_{\lambda},$
where $M_{\lambda}=\left\{ w\in M\ |\  L(0)w=\lambda w\right\} $. Moreover, it is required that
$\dim M_{\lambda}<\infty$ for all $\lambda$  and for any fixed $\lambda_0,$ $M_{\frac{n}{T}+\lambda_0}=0$
for all small enough integers $n.$
\end{definition}

In case $g=1$, we recover the notions of weak, ordinary and admissible
$V$-modules (see \cite{DLM2}).

\begin{definition}A vertex operator algebra $V$ is said to be \emph{$g$-rational}
if the admissible $g$-twisted module category is semisimple. In particular, $V$
is said to be \emph{rational} if $V$ is $1$-rational. \end{definition}

It was proved in \cite{DLM2} that if $V$ is a $g$-rational vertex
operator algebra, then there are only finitely many irreducible admissible
$g$-twisted $V$-modules up to isomorphism and any irreducible
admissible $g$-twisted $V$-module is ordinary.

\begin{definition}A vertex operator algebra $V$ is said to be \emph{regular}
if every weak $V$-module $M$ is a direct sum of irreducible ordinary
$V$-modules. \end{definition}

\begin{definition} A vertex operator algebra $V$ is said to be {\em $C_{2}$-cofinite}
if $V/C_{2}(V)$ is finite dimensional, where $C_{2}(V)={\rm span}\{ u_{-2}v\ |\ u,v\in V\}.$
\end{definition}

A vertex operator algebra $V=\oplus_{n\in\mathbb{Z}}V_{n}$ is said
to be of {\em CFT type} if $V_{n}=0$ for all negative integers $n$ and $V_{0}=\mathbb{C}{\bf 1}$.

\begin{remark} We here recall some known results we shall use extensively.

(1) It was proved in  \cite{DLM2,Li3} that
if $V$ is a $C_{2}$-cofinite vertex operator algebra, then $V$ has only finitely many irreducible admissible modules
up to isomorphism.

(2) Assume that $V$ is rational and $C_{2}$-cofinite. It was proved in \cite{ADJR} that  $V$ is $g$-rational
 for any finite automorphism $g$, and on the other hand, it was proved in \cite{DLM3} that $\lambda_{M}$ is
a rational number for every irreducible $g$-twisted $V$-module $M$.

(3) It is shown in \cite{KL} and \cite{ABD} that for CFT type vertex operator algebras,   regularity is equivalent to rationality and
$C_2$-cofiniteness.
\end{remark}

\begin{definition} Let $M=\bigoplus_{n\in\frac{1}{T}\mathbb{Z}_{+}}M(n)$
be an admissible $g$-twisted $V$-module.  Set
\[
M'=\bigoplus_{n\in\frac{1}{T}\mathbb{Z}_{+}}M\left(n\right)^{*},
\]
the {\em restricted dual}, where $M(n)^{*}=\text{Hom}_{\mathbb{C}}(M(n),\mathbb{C})$.
For $v\in V$, define a vertex operator $Y_{M'}(v,z)$ on $M'$  via
\begin{eqnarray*}
\langle Y_{M'}(v,z)f,u\rangle= \langle f,Y_{M}(e^{zL(1)}(-z^{-2})^{L(0)}v,z^{-1})u\rangle,
\end{eqnarray*}
where $\langle f,w\rangle=f(w)$ is the natural paring $M'\times M\to\mathbb{C}.$
On the other hand, if $M=\oplus_{\lambda\in \C}M_{\lambda}$ is a $g$-twisted $V$-module,
we define $M'=\oplus_{\lambda\in \C}M_{\lambda}^{*}$ and define $Y_{M'}(v,z)$ for $v\in V$ in the same way.
\end{definition}

The same argument of  \cite{FHL} gives:

\begin{proposition} If $(M,Y_{M})$ is an admissible $g$-twisted $V$-module, then
 $(M',Y_{M'})$ carries the structure of
an admissible $g^{-1}$-twisted $V$-module. On the other hand,  if $(M,Y_{M})$ is a $g$-twisted $V$-module, then
 $(M',Y_{M'})$ carries the structure of
a $g^{-1}$-twisted $V$-module. Moreover, $M$ is irreducible if and only if $M'$ is irreducible.
\end{proposition}

In particular, if $(M,Y_{M})$ is a $V$-module, then $(M',Y_{M'})$ is also a $V$-module. A $V$-module $M$ is said to be self-dual if $M$ and $M'$ are isomorphic. A vertex operator algebra $V$ is said to be {\em self-dual} if $V$ and $V'$ are isomorphic $V$-modules.

We shall need the following result from \cite{CM,M2}:

\begin{theorem}\label{CM}
Assume that $V$ is a regular and self-dual vertex operator
algebra of CFT type. Then for any solvable subgroup $G$ of ${\rm Aut}(V)$, $V^{G}$ is a regular  and
self-dual vertex operator algebra of CFT type.
\end{theorem}


Let $g_1,g_2,g_3$ be mutually commuting
automorphisms of $V$ of periods $T_{1}$, $T_2$, $T_3$, respectively.
 In this case, $V$ decomposes into
the direct sum of common eigenspaces for $g_1$ and $g_2$:
\begin{align*}
V=\bigoplus_{0\le j_{1}<T_1,\ 0\le j_{2}<T_2}V^{\left(j_{1},j_{2}\right)},
\end{align*}
where for $j_1,j_2\in \Z$, \emph{
\begin{align}\label{Vij}
V^{\left(j_{1},j_{2}\right)}=\left\{ v\in V\ |\  g_{s}v=e^{2\pi ij_{s}/T_{s}},s=1,2\right\} .
\end{align}}
For any complex number $\alpha$, we define
$$(-1)^{\alpha}=e^{\alpha \pi i}.$$
Now, we define intertwining operators among weak $g_{s}$-twisted modules
$(M_{s},Y_{M_{s}})$ for $s=1,2,3$.

\begin{definition} \label{Intertwining operator for twisted modules}
An {\em intertwining operator of type $\binom{M_3}{M_{1}\ M_{2}}$}
associated with the given data is a linear map
\begin{align*}
\mathcal{Y}(\cdot,z): \  \  &  M_{1}\to\left(\text{Hom}(M_{2},M_{3})\right)\{ z\}\nonumber\\
&w\mapsto\mathcal{Y}\left(w,z\right)=\sum_{n\in\mathbb{C}}w_{n}z^{-n-1}
\end{align*}
such that for any $w^{1}\in M_{1},\ w^{2}\in M_{2}$ and for any fixed $c\in\mathbb{C}$,
\[
w_{c+n}^{1}w^{2}=0\   \   \mbox{  for $n\in\mathbb{Q}$ sufficiently large},
\]
\begin{align}
 & z_{0}^{-1}\left(\frac{z_{1}-z_{2}}{z_{0}}\right)^{j_{1}/T_{1}}\delta\left(\frac{z_{1}-z_{2}}{z_{0}}\right)
 Y_{M_{3}}(u,z_{1})\mathcal{Y}(w,z_{2})\nonumber \\
 &\   \  \   \  \
  -z_{0}^{-1}\left(\frac{z_{2}-z_{1}}{-z_{0}}\right)^{j_{1}/T_{1}}\delta\left(\frac{z_{2}-z_{1}}{-z_{0}}\right)
  \mathcal{Y}(w,z_{2})Y_{M_{2}}(u,z_{1})\nonumber \\
 =\  \  &z_{2}^{-1}\left(\frac{z_{1}-z_{0}}{z_{2}}\right)^{-j_{2}/T_{2}}\delta\left(\frac{z_{1}-z_{0}}{z_{2}}\right)
 \mathcal{Y}\left(Y_{M_{1}}(u,z_{0})w,z_{2}\right)\label{Twisted Intertwining}
\end{align}
 on $M_{2}$ for $u\in V^{(j_{1},j_{2})}$ with $j_1,j_2\in \Z$ and
$w\in M_{1}$,  and
\[
\frac{d}{dz}\mathcal{Y}(w,z)=\mathcal{Y}(L(-1)w,z).
\]
All intertwining operators of type $\binom{M_3}{M_{1}\ M_{2}}$
form a vector space, which we denote by \emph{$I_{V}\binom{M_3}{M_{1}\ M_{2}}$.}
Set $N_{M_{1} M_{2}}^{M_{3}}=\dim I_{V}\binom{M_3}{M_{1}\ M_{2}}$,
which is called the \textit{fusion rule}.
\end{definition}

\begin{remark}\label{r2.14}
As it was proved in \cite{Xu},  if there are weak $g_{s}$-twisted modules
$(M_{s},Y_{M_{s}})$ for $s=1,2,3$ such that $N_{M_{1}\ M_{2}}^{M_{3}}>0$,
then $g_{3}=g_{1}g_{2}$. In view of this, we shall always assume that $g_{3}=g_{1}g_{2}$.
\end{remark}

By the same arguments in Chapter 7 of \cite{DL} for formulas
(7.24) and (7.11) we have:

\begin{proposition}\label{p2.14}
The twisted Jacobi identity (\ref{Twisted Intertwining})
is equivalent to the {\em generalized commutativity:} For $u\in V^{(j_{1},j_{2})}$
and $w\in M_{1}$, there is a positive integer $n$ such that
\[
(z_{1}-z_{2})^{n+\frac{j_{1}}{T_{1}}}Y_{M_{3}}(u,z_{1}){\cal Y}(w,z_{2})
=(-z_{2}+z_{1})^{n+\frac{j_{1}}{T_{1}}}{\cal Y}(w,z_{2})Y_{M_{2}}(u,z_{1})
\]
for $u\in V^{(j_{1},j_{2})},\ w\in M_{1}$ and {\em generalized associativity:}
For $u\in V^{(j_{1},j_{2})},$ $w\in M_{1}$ and $w_{2}\in M_{2}$,
there is a positive integer $n$ depending on $u$ and $w_{2}$ only
such that
\[
(z_{0}+z_{2})^{n+\frac{j_{2}}{T_{2}}}Y_{M_{3}}(u,z_1){\cal Y}(w,z_{2})w_{2}
=(z_{2}+z_{0})^{n+\frac{j_{2}}{T_{2}}}{\cal Y}(Y_{M_{1}}(u,z_{0})w,z_{2})w_{2}.
\]
\end{proposition}

Let $V_1$ and $V_2$ be vertex operator algebras and let $\sigma_1, \sigma_2$ be finite order automorphisms of $V_1, V_2$,
respectively. Let $W_1$ be a $V_1$-module,
$W_2,W_3$ $\sigma_1$-twisted $V_1$-modules, and let $N_1$ be a $V_2$-module,
$N_2,N_3$ $\sigma_2$-twisted $V_2$-modules. Note that $W_1\otimes N_1$ is a $V_1\otimes V_2$-module, while
$W_2\otimes N_2$ and $W_3\otimes N_3$ are $\sigma$-twisted $V_1\otimes V_2$-modules with
$\sigma=\sigma_1\otimes \sigma_2$.
For any intertwining operator $\mathcal{Y}_{1}(\cdot,z)$
of type $\binom{W_3}{W_1\ W_2}$ and for any intertwining operator $\mathcal{Y}_{2}(\cdot,z)$
of type $\binom{N_3}{N_1\ N_2}$, define a linear map
$$(\mathcal{Y}_1\otimes \mathcal{Y}_2)(\cdot,z):\  W_1\otimes N_1\rightarrow
 \left(\text{Hom} (W_2\otimes N_2,W_{3}\otimes N_3)\right)\{z\}$$
by
\begin{align}
(\mathcal{Y}_1\otimes \mathcal{Y}_2)(w_1\otimes n_1,z)(w_2\otimes n_2)
=\mathcal{Y}_{1}(w_1,z)w_2\otimes \mathcal{Y}_{2}(n_1,z)n_2
\end{align}
for $w_1\in W_1,\ w_2\in W_2,\ n_1\in N_1,n_2\in N_2$.
As in \cite{ADL}, by using the generalized commutativity and associativity, it is straightforward to show that
$(\mathcal{Y}_1\otimes \mathcal{Y}_2)(\cdot,z)$ is an intertwining operator of type
$\binom{W_3\otimes N_3}{W_1\otimes N_1\  W_2\otimes N_2}$. Then we have a linear map
\begin{align}
\psi: \  \  I_{V_1}\binom{W_3}{W_1\ W_2}\otimes  I_{V_2}\binom{N_3}{N_1\ N_2} &\rightarrow
I_{V_1\otimes V_2}\binom{W_3\otimes N_3}{W_1\otimes N_1\  W_2\otimes N_2}\nonumber\\
(\mathcal{Y}_{1},\mathcal{Y}_2)&\mapsto \mathcal{Y}_{1}\otimes \mathcal{Y}_2.
\end{align}

The same arguments of \cite{ADL} in proving Theorem 2.10 (with minor necessary modifications) give:

\begin{theorem}\label{ADL-twisted-version}
Let $V_1,V_2,\ \sigma_1,\sigma_2, \ W_1,W_2,W_3, N_1,N_2, N_3$ be given as above.
Then the linear map $\psi$ is one-to-one. Furthermore, if either
$$\dim I_{V^1}\binom{W_3}{W_1\ W_2}<\infty, \  \mbox{ or }  \dim I_{V^2}\binom{N_3}{N_1\ N_2}<\infty,$$
then $\psi$ is a linear isomorphism.
\end{theorem}

Set ${\mathbb{D}}=\C^{2}\setminus\{(z,z),(z,0),(0,z)\ |\  z\in\C\}.$ Then
we have:

 \begin{proposition}\label{p2.15}
 Let $V$ be rational and $C_{2}$-cofinite,
and let $G\subset {\rm Aut}(V)$ be a finite abelian group. Then in the presence of other conditions in Definition \ref{Intertwining operator for twisted modules}, the Jacobi identity
(\ref{Twisted Intertwining}) is equivalent to the following properties for $g_{1},g_{2}\in G$ and $v\in V^{(j_{1},j_{2})},w\in M_{1},w_{2}\in M_{2},w_{3}'\in M_{3}'$:

(a) The formal series
\[
\langle w_{3}',Y_{M_{3}}(v,z_{1}){\cal Y}(w,z_{2})w_{2}\rangle(z_{1}-z_{2})^{j_{1}/T_{1}}z_{1}^{j_{2}/T_{2}}
\]
converges to a unique multi-valued analytic function $f(z_{1},z_{2})=\frac{\sum_{i}h_{i}(z_{1},z_{2})z_{2}^{r_{i}}}{z_{1}^{r}z_{2}^{s}(z_{1}-z_{2})^{t}}$
for some $h_{i}(z_{1},z_{2})\in\C[z_{1},z_{2}],$ $r_{i}\in\Q$ and
$r,s,t\in\Z$ in the domain $|z_{1}|>|z_{2}|>0.$

(b) The formal series
\[
\langle w_{3}',{\cal Y}(w,z_{2})Y_{M_{2}}(v,z_{1})w_{2}\rangle(-z_{2}+z_{1})^{j_{1}/T_{1}}z_{1}^{j_{2}/T_{2}}
\]
converges to $f(z_{1},z_{2})$ in the domain $|z_{2}|>|z_{1}|>0.$

(c) The formal series
\[
\langle w_{3}',{\cal Y}\left(Y_{M_{1}}\left(v,z_{1}-z_{2}\right)w,z_{2}\right)w_{2}\rangle(z_{1}-z_{2})^{j_{1}/T_{1}}z_{1}^{j_{2}/T_{2}}
\]
converges to $f(z_{1},z_{2})$ in the domain $|z_{2}|>|z_{1}-z_{2}|>0.$

\end{proposition} \begin{proof} Since $V^{G}$ is rational by Theorem
\ref{CM},  from \cite{DLM3} each $M_{i}$ is a direct sum of finitely many irreducible
$V^{G}$-modules whose weights are rational. Without
loss we can assume that $w,w_{2},w_{3}'$ are vectors in some irreducible
$V^{G}$-modules. Then there exists a rational number $r$ such that
the powers of $z_{2}$ in the formal power series $\langle w_{3}',Y_{M_{3}}(v,z_{1}){\cal Y}(w,z_{2})w_{2}\rangle$
lie in $r+\Z.$ The same is true for the corresponding terms in (b)
and (c). Now results follow from Proposition \ref{p2.14} and the
proofs of Propositions 7.6, 7.8 and 7.9 of \cite{DL}. \end{proof}

Now we recall a notion of tensor product (see \cite{HL1-2}, \cite{Li2}, \cite{DLM1}).

\begin{definition}\label{d2.13}
Let $g_1,g_2$ be commuting finite order automorphisms of a vertex operator algebra $V$ and
let $M_i$ be a $g_i$-twisted $V$-module for $i=1,2$. A {\em tensor product} for the ordered
pair $\left(M_{1},M_{2}\right)$ is a pair $(M,F(\cdot,z))$
where $M$ is a weak $g_1g_{2}$-twisted $V$-module and $F\left(\cdot,z\right)$ is an intertwining
operator  of type $\binom{M}{M_{1}\ M_{2}}$
such that the following universal property holds: For any weak $g_1g_{2}$-twisted
$V$-module $W$ and for any intertwining operator $I(\cdot,z)$
of type $\binom{W}{M_{1}\ M_{2}},$ there exists a unique
weak twisted $V$-module homomorphism $\psi$ from $M$ to $W$ such that $I(\cdot,z)=\psi\circ F(\cdot,z).$
\end{definition}

We shall use $M_{1}\boxtimes^{V}M_{2}$ to denote a generic tensor product module
of $M_{1}$ and $M_{2}$, provided that the existence is verified.

\begin{remark}\label{tautological}
We here recall Huang-Lepowsky's tautological construction of a tensor product (see \cite{HL1-2}).
Suppose that $V$ is a vertex operator algebra with a finite order automorphism $\tau$ such that $V$ is $\tau$-rational, namely,
every $\tau$-twisted $V$-module is completely reducible.  Let $W$ be a $V$-module and $T$ a $\tau$-twisted $V$-module
such that $\dim I_{V}\binom{P}{W\ T}<\infty$  for every irreducible $\tau$-twisted $V$-module $P$.
Let $\{P_{\alpha}\ |\ \alpha\in \mathcal{A}\}$ be a complete set of equivalence class representatives of irreducible $\tau$-twisted $V$-modules. Set
\begin{align}
M=\oplus_{\alpha\in \mathcal{A}} I\binom{P_{\alpha}}{W\ T}^{*}\otimes P_{\alpha}
=\oplus_{\alpha\in \mathcal{A}}\text{Hom}\left(I\binom{P_{\alpha}}{W\ T},P_{\alpha}\right),
\end{align}
a $\tau$-twisted $V$-module. For $\alpha\in \mathcal{A}$, define $\mathcal{Y}^{\alpha}(w,z)t\in \text{Hom}\left(I\binom{P_{\alpha}}{W\ T},P_{\alpha}\right)\{z\}$ by
$$(\mathcal{Y}^{\alpha}(w,z)t)(I)=I(w,z)t\   \   \   \mbox{ for }I\in I\binom{P_{\alpha}}{W\ T}.$$
Set $\mathcal{Y}(\cdot,z)=\sum_{\alpha\in \mathcal{A}}\mathcal{Y}^{\alpha}(\cdot,z)$.
Then $(M,\mathcal{Y})$ is a tensor product of $W$ and $T$.
\end{remark}

The following is a simple property of fusion rules (numbers) (see \cite{FHL}, \cite{HL1-2}, \cite{DLM1}:

\begin{lemma}
Let $g_1,g_2$ be commuting finite order automorphisms of a vertex operator algebra $V$ and
let $M_i$ be a $g_i$-twisted $V$-module for $i=1,2,3$ with $g_3=g_1g_2$.
Then $N_{M_{1}\ M_{2}}^{M_{3}}=N_{M_{2}\ M_{1}}^{M_{3}}.$
\end{lemma}

\section{Tensor products of twisted modules}

Throughout this section, we assume that $V$ is a regular and self-dual vertex operator algebra of CFT type
and $G$ is a finite abelian automorphism group of $V$.
The main goal of this section is to prove that for any $g,h\in G$ and for any $g$-twisted
$V$-module $M$ and $h$-twisted $V$-module $N,$ a tensor product of
$M$ and  $N$ exists and associativity holds. These results will be used
later to determine the tensor product of an untwisted module with a twisted
module for the study on permutation orbifolds.

Denote by ${\cal C}_{V}$ the category of ordinary $V$-modules and  by ${\cal C}_{V^G}$ the category of ordinary $V^G$-modules.
From \cite{H}, both ${\cal C}_{V}$ and ${\cal C}_{V^{G}}$ are modular tensor
categories.  Furthermore, from  \cite{KO} and \cite{CKM},  $V$
is a commutative associative algebra in ${\cal C}_{V^{G}}$ as
$V=\oplus_{\chi\in {\rm irr}(G)}V^{\chi}$, where ${\rm irr}(G)$ is the set of irreducible characters of $G$ and $V^{\chi}$ are irreducible $V^{G}$-modules (cf. \cite{DJX}), i.e., simple objects in
${\cal C}_{V^{G}}.$

Recall the following definition from \cite{KO} and \cite{CKM} (see Proposition 3.46 of \cite{CKM}):

\begin{definition}\label{def-rep(V)}
Denote by ${\rm Rep}(V)$ the subcategory of ${\cal C}_{V^{G}}$ consisting of every $V^{G}$-module
$W$ together with a $V^{G}$-intertwining operator $Y_W(\cdot,z)$ of type $\binom{W}{V\ W}$
such that the following conditions are satisfied:

1. (Associativity) For any $u,v\in V,$ $w\in W$ and $w'\in W'$, the formal series
\[
\langle w',Y_{W}(u,z_{1})Y_{W}(v,z_{2})w\rangle
\]
and
\[
\langle w',Y_{W}(Y(u,z_{1}-z_{2})v,z_{2}))w\rangle
\]
converge on the domains $|z_{1}|>|z_{2}|>0$ and
 $|z_{2}|>|z_{1}-z_{2}|>0$, respectively, to multivalued analytic functions
which coincide on their common domain.

2. (Unit) $Y_{W}({\bf 1},z)=Id_{W}.$
\end{definition}

Recall from \cite{KO} (see also \cite{CKM}) that there is a categorical
tensor product functor $\boxtimes_{V}$ in the category of ${\rm Rep}(V)$, which is associative.
We denote by $M\boxtimes_VN$ the tensor product of $M$ and $N$ in ${\rm Rep}(V)$.
Then for any two $V^G$-modules
$M,N$ in ${\rm Rep}(V),$ $M\boxtimes_{V}N$ is a quotient of $M\boxtimes^{V^G} N.$ Moreover, ${\rm Rep}(V)$ is a fusion category. In particular,
${\rm Rep}(V)$ is a semisimple category with finitely many inequivalent simple objects.

From the generalized associativity for a twisted module (see Proposition \ref{p2.14}), we immediately have:

\begin{lemma}\label{l3.1}
If $W$ is a $g$-twisted $V$-module with $g\in G$, then $W$ is an object of ${\rm Rep}(V).$ Furthermore, if
$W_i$ is a $g_i$-twisted $V$-module with $g_i\in G$ for $i=1,2$, then $W_1$ and $W_2$ are equivalent objects
in ${\rm Rep}(V)$ if and only if $g_1=g_2$ and $W_1\simeq W_2$ as twisted $V$-modules.
\end{lemma}

On the other hand, we have the following lemma (see also \cite{K1,K2}).

\begin{lemma}\label{l3.1'}
If $W$ is a simple object in ${\rm Rep}(V),$ then $W$ is an irreducible $g$-twisted $V$-module  for some
$g\in G.$
\end{lemma}

\begin{proof}  We need to use the Frobenius-Perron dimension in a fusion category 
(see  \cite{BK}, \cite{ENO},  \cite{DMNO}). Let $\cc$ be a fusion category and
let $K(\cc)$ denote the Grothendieck ring of $\cc.$  According to \cite{ENO},
there exists
a unique ring homomorphism $\FPdim_{\cc}: K(\cc)\to \R$ such that $\FPdim_{\cc}(X) > 0$ for all
nonzero $X\in\cc$, where
$\FPdim_{\cc}(X)$ is called the Frobenius-Perron dimension of $X\in\cc.$
Furthermore, one has the Frobenius-Perron dimension for the category $\cc$:
$$ \FPdim (\cc)=\sum_{X\in {\mathcal O}(\cc)}\FPdim_{\cc}(X)^2,$$
where ${\mathcal O}(\cc)$ denotes the equivalence classes of the simple objects in $\cc.$
Note that both ${\cal C}_{V^{G}}$ and ${\rm Rep}(V)$ are fusion categories.
It follows from
 \cite{DMNO} and \cite{ENO} that
\begin{equation}\label{equationnew}(\FPdim_{{\cal C}_{V^{G}}} V) (\FPdim ({\rm Rep}(V))) = \FPdim ({\cal C}_{V^{G}}).
\end{equation}

We also need quantum dimensions and global dimensions from \cite{DJX} and \cite{DRX}.
Let $M$ be a $g$-twisted $V$-module with $g\in G$ and set ${\rm ch}_qM={\rm tr}|_Mq^{L(0)-c/24}.$
Then ${\rm ch}_qM$ converges to a holomorphic function $\chi_M(\tau)$ on the upper half plane with $q=e^{2\pi i\tau}$ (see \cite{Z}, \cite{DLM3}). The quantum dimension of $M$ over $V$ is defined as
$$
\qdim_{V}M=\lim_{y\to 0}\frac{\chi_M(iy)}{\chi_V(iy)}.$$
From  \cite{ADJR} we have
$$\qdim_{V}M=\FPdim_{{\rm Rep}(V)}M,\  \mbox{ and }\   \qdim_{V^G}W= \FPdim_{{\cal C}_{V^{G}}} W$$
 for any $V^G$-module $W$.
Denote the equivalence classes of irreducible $g$-twisted $V$-modules by ${\cal M}(g).$
The global dimension of $V$ is defined as
$${\rm glob}(V)=\sum_{M\in {\cal M}(1)}(\qdim_VM)^2.$$
It follows that ${\rm glob}(V^G)=\FPdim(\cc_{V^G}).$
Recall from \cite{DJX} and \cite{DRX}  that
$$\qdim_{V^G}V=o(G),$$
 $${\rm glob}(V^G)=o(G)^2{\rm glob}(V),$$
 $${\rm glob}(V)=\sum_{M\in {\cal M}(g)}(\qdim_VM)^2$$
for any $g\in G.$ Then we have
$$\FPdim_{{\cal C}_{V^{G}}} V= \qdim_{V^G}V=o(G),\   \   \   \
\FPdim(\cc_{V^G})={\rm glob}(V^G)=o(G)^2{\rm glob}(V).$$
Combining these relations with equation (\ref{equationnew}) we obtain
$$\FPdim ({\rm Rep}(V))=o(G){\rm glob}(V).$$

Since for any $g\in G$,  every irreducible $g$-twisted module is a simple object in ${\rm Rep}(V)$ by Lemma \ref{l3.1}, we get
$$\FPdim {\rm Rep}(V)\geq \sum_{g\in G}\sum_{M\in {\cal M}(g)}(\qdim_VM)^2=o(G){\rm glob}(V)=\FPdim {\rm Rep}(V).$$
Thus  $$\FPdim {\rm Rep}(V)= \sum_{g\in G}\sum_{M\in {\cal M}(g)}(\qdim_VM)^2,$$
which implies that every simple object in ${\rm Rep}(V)$ is an irreducible $g$-twisted $V$-module for some $g\in G.$
\end{proof}

\begin{remark} Note that Lemma \ref{l3.1'} holds for any (not necessarily abelian) finite automorphism group $G$
  as long as $V^G$ is rational and $C_2$-cofinite.
\end{remark}


Let $g,h\in G$. For any $g$-twisted module $M$ and $h$-twisted module $N$, as they are objects in ${\rm Rep}(V)$
by Lemma \ref{l3.1},  $M\boxtimes_{V}N$ exists in ${\rm Rep}(V).$ We
will show next that the tensor product in the sense of Definition \ref{d2.13}
exists and is isomorphic to $M\boxtimes_{V}N.$
We first prove that a categorical intertwining
operator is the same as an intertwining operator in the sense of Definition \ref{Intertwining operator for twisted modules}.

The following is a generalization of Theorem 3.53 in \cite{CKM} where $g_{i}=1$ for $i=1,2,3$
with a similar proof:

\begin{theorem}\label{inter1}
Let $g_1,g_2, g_3$ be commuting finite order automorphisms of $V$ and
let $M_i$ be a $g_i$-twisted $V$-module for $i=1,2,3.$
Assume
\begin{align*}
{\cal Y}(\cdot,z)\cdot:\  M_{1}\otimes M_{2} & \longrightarrow M_{3}\{z\}\\
w_{1}\otimes w_{2} & \longmapsto{\cal Y}(w_{1},z)w_{2}=\sum_{n\in\mathbb{C}}(w_{1})_{n}w_2z^{-n-1}
\end{align*}
is a bilinear map satisfying the following conditions:
\begin{enumerate}
\item \emph{Lower truncation:} For any $w_{1}\in M_{1}$, $w_{2}\in M_{2}$
and $\alpha \in\C$, $(w_{1})_{\alpha+m}w_{2}=0$ for sufficiently large $m\in\mathbb{N}$.
\label{item:complexLower}
\item The \emph{$L(-1)$-derivative formula}: $[L(-1),\cY(w_{1},z)]=\dfrac{d}{dx}\cY(w_{1},z)=\cY(L(-1)w_{1},z)$
for $w_{1}\in M_{1}$. \label{item:complexL0}
\item The \emph{$L(0)$-bracket formula:} $[L(0),\cY(w_{1},z)]=z\dfrac{d}{dz}\cY(w_{1},z)+\cY(L(0)w_{1},z)$
for $w_{1}\in M_{1}$. \label{item:complexL-1}
\item For any $w_{3}'\in M_{3}'$, $w_{2}\in M_{2}$, $v\in V$, $w_{1}\in M_{1}$,
the following series define multivalued analytic functions on the
indicated domains, which coincide on the respective
intersections of their domains: \label{item:complexMain}
\begin{align}
\langle w_{3}',Y_{M_{3}}(a,z_{1})\cY(w_{1},z_{2})w_{2}\rangle & \quad\quad\mbox{ on }|z_{1}|>|z_{2}|>0,\label{eqn:prodYI}\\
\langle w_{3}',\cY(Y_{M_{1}}(a,z_{1}-z_{2})w_{1},z_{2})w_{2}\rangle & \quad\quad\mbox{ on }|z_{2}|>|z_{1}-z_{2}|>0,\label{eqn:assocIY}\\
\langle w_{3}',\cY(w_{1},z_{2})Y_{M_{2}}(a,z_{1})w_{2}\rangle & \quad\quad  \mbox{ on }|z_{2}|>|z_{1}|>0.\label{eqn:prodIY}
\end{align}
Specifically, the principal branches (that is, $z_{2}=e^{\log z_{2}}$)
of the first and second multivalued functions yield single-valued
functions which coincide on the simply connected domain
\[
S_{1}:=\lbrace(z_{1},z_{2})\in\C^{2}\,\vert\,\mathrm{Re}\,z_{1}>\mathrm{Re}\,z_{2}>\mathrm{Re}(z_{1}-z_{2})>0,\,\mathrm{Im}\,z_{1}>\mathrm{Im}\,z_{2}>\mathrm{Im}(z_{1}-z_{2})>0\rbrace
\]

and the principal branches of the second and third multivalued functions
yield single-valued functions which coincide on the simply connected domain
\[
S_{2}:=\lbrace(z_{1},z_{2})\in\C^{2}\,\vert\,\mathrm{Re}\,z_{2}>\mathrm{Re}\,z_{1}>\mathrm{Re}(z_{2}-z_{1})>0,\,\mathrm{Im}\,z_{2}>\mathrm{Im}\,z_{1}>\mathrm{Im}(z_{2}-z_{1})>0\rbrace.
\]

\item There exists a multivalued analytic function $g$ defined on $\{(z_{1},z_{2})\in\mathbb{C}^{2}\,|\,z_{1},z_{2},z_{1}-z_{2}\neq0\}$
that restricts to the three multivalued functions above on their respective
domains. \label{item:complexMainFunction}
\end{enumerate}
Then $\cY$ is an intertwining operator of type $\binom{M_{3}}{M_{1}\,M_{2}}$
in the sense of Definition \ref{Intertwining operator for twisted modules}.
\end{theorem}

\begin{proof} The proof here is a slight modification of that of Theorem 3.53 in \cite{CKM}.
First, fix $z_{2}\in\mathbb{C}^{\times}$ with  $\mathrm{Re}\,z_{2},\mathrm{Im}\,z_{2}>0$.
Then $f_{z_{2}}(z_{1}):=g(z_{1},e^{\log z_{2}})z_{1}^{j_{1}/T_{1}}(z_{1}-z_{2})^{j_{2}/T_{2}}$
yields a possibly multivalued analytic function of $z_{1}$ defined
on $\{z_{1}\in\mathbb{C}^{\times}\,|\,z_{1}\neq z_{2}\}$.
 However,
note that $f_{z_{2}}$ has single-valued restrictions to each of $\{z_{1}\in\mathbb{C}^{\times}\,|\,|z_{1}|>|z_{2}|\}$,
$\{z_{1}\in\mathbb{C}^{\times}\,|\,|z_{2}|>|z_{1}-z_{2}|>0\}$ and
$\{z_{1}\in\mathbb{C}^{\times}\,|\,|z_{2}|>|z_{1}|\}$, and these
restrictions coincide on certain simply-connected subsets of the
intersections of these domains. Thus these restrictions define a single-valued
analytic function on the union of these domains. This implies that
$f_{z_{2}}$ has a single-valued restriction $\widetilde{f}_{z_{2}}$
to $\C^{\times}\setminus\lbrace z_{2}\rbrace$ with possible poles
at $0,z_{2},\infty.$ Thus $\widetilde{f}_{z_{2}}$ is a rational
function. It follows from Proposition \ref{p2.15} that ${\cal Y}$
satisfies the Jacobi identity and hence it is an intertwining operator of type
$\binom{M_{3}}{M_{1}\,M_{2}}$ in the sense of Definition \ref{Intertwining operator for twisted modules}.
\end{proof}

\begin{theorem}\label{t3.2}
Let $g_{1},g_2,g_3\in G$ and let $M_{i}$ be $g_{i}$-twisted $V$-modules for $i=1,2,3$.
Then the space of intertwining operators
of type $\binom{M_3}{M_{1}\ M_{2}}$ in the sense of Definition
\ref{Intertwining operator for twisted modules} is linearly isomorphic
to the space of the categorical ${\rm Rep}(V)$ intertwining operators
of type $\binom{M_3}{M_{1}\ M_{2}}.$
Furthermore, a tensor product module $M_{1}\boxtimes^{V}M_{2}$ exists
and $M_{1}\boxtimes^{V}M_{2}$ and $M_{1}\boxtimes_{V}M_{2}$
are isomorphic $g_1g_2$-twisted $V$-modules.
\end{theorem}

\begin{proof} By Theorem 3.44 of \cite{CKM}, the space of the
categorical ${\rm Rep}(V)$ intertwining operators of type $\binom{M_3}{M_{1}\ M_{2}}$
is isomorphic to the space of bilinear maps
\begin{align*}
{\cal Y}(\cdot,z)\cdot :\   M_{1}\otimes M_{2} & \longrightarrow M_{3}\left\{ z\right\} \\
w_{1}\otimes w_{2} & \longmapsto{\cal Y}(w_{1},z)w_{2}=\sum_{n\in\mathbb{C}}(w_{1})_{n}w_2z^{-n-1}
\end{align*}
satisfying the conditions 1-5 in Theorem \ref{inter1}. By Proposition \ref{p2.15} and Theorem
\ref{inter1}, the latter space is the same as the space
of intertwining operators of type $\binom{M_3}{M_{1}\ M_{2}}$
in the sense of Definition \ref{Intertwining operator for twisted modules}. It follows from  Remark \ref{r2.14} that the space of the
categorical ${\rm Rep}(V)$ intertwining operators of type $\binom{M_3}{M_{1}\ M_{2}}$ is zero if $g_1g_2\ne g_3.$ So we assume $g_1g_2=g_3$ for the rest of the proof.

Let $\{M_{3}^{1},\dots, M_3^r\}$
be a complete set of equivalence class representatives of irreducible $g_{3}$-twisted $V$-modules.
From Lemma \ref{l3.1}, $M_3^{1},\dots,M_3^r$ are inequivalent  simple objects in  ${\rm Rep}(V)$.
From the first part,  the fusion rules $N_{M_{1},M_{2}}^{{M_{3}^{i}}}$ of the indicated types
 in the sense
of Definition \ref{Intertwining operator for twisted modules} and in ${\rm Rep}(V)$ are
the same.
As $M_{1}\boxtimes_{V}M_{2}$ is an object in ${\rm Rep}(V)$ (consisting of some $V^{G}$-modules), the
fusion numbers $N_{M_{1},M_{2}}^{{M_{3}^{i}}}$ are finite.
For $1\le i\le r$,  set $N_{i}=N_{M_{1},M_{2}}^{{M_{3}^{i}}}$ for short.
Set
$$W=\oplus_{i=1}^r(\C^{N_i}\otimes M_3^{i}),$$
which is a $g_{3}$-twisted $V$-module and hence an object in ${\rm Rep}(V)$.
It follows that there is an isomorphism in the category ${\rm Rep}(V)$ from $M_{1}\boxtimes_{V}M_{2}$ to $W$.

For $1\le i\le r$, let $\cY_{1}^{i},...,\cY_{N_{i}}^{i}$
be a basis of $I_{V}\binom{M_{3}^{i}}{M_{1}\ M_{2}}$.
Then we have intertwining operators $e_s\otimes \cY_{j}^{i}$ for $1\le i\le r,\ 1\le s,j\le N_i$
of type $\binom{W}{M_1\ M_2}$, where $e_1,\dots, e_{N_i}$ are the standard basis vectors of $\C^{N_i}$.
Set
\begin{align}\label{def-F}
F(\cdot,z)=\sum_{i=1}^{r}\sum_{s,j=1}^{N_{i}} e_s\otimes \cY_{j}^{i}(\cdot,z).
\end{align}
We now prove that $(W,F)$ is a tensor product of $M_{1}$ and $M_{2}$
in the sense of Definition \ref{d2.13}. Let $M$ be any $g_{3}$-twisted
$V$-module and $I\in I_{V}\binom{M}{M_{1}\ M_{2}}$. We need to show that there
is a unique $V$-homomorphism $\psi:W\to M$ such that $I=\psi\circ\cY.$
Since $V$ is $g_3$-rational, $M$ is isomorphic to a direct sum of $M_3^i$ for $1\le i\le r$.
Thus $M=\oplus_{i=1}^{r}\text{Hom}_V(M_3^i,M)\otimes M_3^i$.
 It can be readily seen that
$$I_V\binom{M}{M_1\ M_2}=\oplus_{i=1}^{r}  \text{Hom}_V(M_3^i,M)\otimes I_V\binom{M_3^i}{M_1\ M_2}.$$
Then $I=\sum_{i=1}^r\sum_{j=1}^{N_{i}}\psi_{i,j}\otimes \cY_{j}^{i}$
where $\psi_{i,j}\in \text{Hom}_V(M_3^i,M)$. For $1\le i\le r,\ 1\le s\le N_i$, let $p_{i,s}$ be the projection of $M$
onto $e_s\otimes M_{3}^i$, a $V$-homomorphism, and let $\pi_{i,s}$ be the  identification map
$e_s\otimes M_{3}^i\simeq M_3^i$.
Set $\psi=\sum_{i=1}^r\sum_{j,s=1}^{N_i}\psi_{i,j}\pi_{i,s}p_{i,s}$, a $V$-homomorphism from $W$ to $M$.
It is clear that  $I=\psi\circ F.$ Thus $(W,F)$ is a tensor product of $M_{1}$ and $M_{2}$
in the sense of Definition \ref{d2.13}. Consequently, we have $M_{1}\boxtimes^{V}M_{2}\simeq M_{1}\boxtimes_{V}M_{2}.$
\end{proof}

Since the tensor product functor $\boxtimes_{V}$ is associative, we have:

\begin{corollary}\label{t3.5}
Let $M_{i}$ be a $g_{i}$-twisted $V$-module with $g_{i}\in G$ for $i=1,2,3$. Then
$$(M_{1}\boxtimes_{V}M_{2})\boxtimes_{V}M_{3}\simeq M_{1}\boxtimes_{V}(M_{2}\boxtimes_{V}M_{3}).$$
\end{corollary}


\section{\label{subsec:Space-of-generalized-general results}Weak $V$-module $\mathcal{H}(M_{2},M_{3})$
of generalized intertwining operators}
Let $V$ be a vertex operator algebra and let $\tau\in\text{Aut }(V)$
of period $T$. In this section, we first prove that for any $\tau$-twisted
$V$-modules $M_{2}$ and $M_{3}$, there is a weak module structure
on the space $\mathcal{H}(M_{2},M_{3})$ of generalized intertwining
operators. Furthermore, we prove that for any weak $V$-module $M_{1}$,
the space $\text{Hom}_{V}\left(M_{1},\mathcal{H}(M_{2},M_{3})\right)$
of $V$-homomorphisms from $M_{1}$ to $\mathcal{H}(M_{2},M_{3})$
is canonically isomorphic to $I_{V}\binom{M_{3}}{M_{1}M_{2}}$. Hence
giving a $V$-homomorphism from $M_{1}$ to $\mathcal{H}(M_{2},M_{3})$
is equivalent to giving an intertwining operator of type $\binom{M_{3}}{M_{1}M_{2}}$.

As before, for $r\in\Z$ set
$$V^r=\left\{ v\in V\ |\  \tau v=\eta_T^rv\right\},$$
where
$$\eta_{T}=e^{2\pi i/T}$$ is
the principal primitive $T$-th root of unity.

\begin{definition}
Let $(M_{2},Y_{M_{2}})$ and $(M_{3},Y_{M_{3}})$
be weak $\tau$-twisted $V$-modules.
A \emph{generalized intertwining operator} from $M_{2}$ to $M_{3}$
is an element $\phi(x)=\sum_{\alpha\in\mathbb{C}}\phi_{\alpha}x^{-\alpha-1}\in(\text{Hom}(M_{2},M_{3}))\{x\}$
satisfying the following conditions:

(G1) \  \ For any $\alpha\in\mathbb{C}$, $w\in M_{2}$, $\phi_{\alpha+n}w=0$
for $n\in\mathbb{Z}$ sufficiently large;

(G2) $\  \  \left[L(-1),\phi(x)\right]=\frac{d}{dx}\phi(x)$;

(G3)  \  \  For any $v\in V$, there exists a nonnegative integer $k$ such
that
\begin{equation}\label{weak-comm-phi(x)}
(x_{1}-x_{2})^{k}Y_{M_{3}}(v,x_{1})\phi(x_{2})=(x_{1}-x_{2})^{k}\phi(x_{2})Y_{M_{2}}(v,x_{1}).
\end{equation}
Denote by $\mathcal{H}(M_{2},M_{3})$ the space of all
\emph{generalized intertwining operators} from $M_{2}$ to $M_{3}$.
\end{definition}

Let $\phi(x)\in \mathcal{H}(M_{2},M_{3})$ and let $u\in V^{r}$
 with $0\le r<T$.  For each $n\in \Z$, we define
\begin{align}\label{unH-phi}
&u_{n}^{\mathcal{H}}\phi(x)\nonumber\\
=\ &\text{Res}_{x_{1}}\sum_{j\ge 0}\binom{-\frac{r}{T}}{j}x^{-\frac{r}{T}-j}x_1^{\frac{r}{T}}
\left((x_{1}-x)^{n+j}Y_{M_{3}}(u,x_{1})\phi(x)
 -(-x+x_{1})^{n+j}\phi(x)Y_{M_{2}}(u,x_{1})\right),
 \end{align}
 an element of $(\text{Hom}(M_{2},M_{3}))\{x\}$.

 \begin{remark}\label{explanations}
 Here are some explanations about this definition.
First, recall that  for any weak $\tau$-twisted $V$-module $(W,Y_W)$, we have
$$z^{\frac{r}{T}}Y_W(u,z)w\in W((z))\   \   \   \   \mbox{for }u\in V^r,\ r\in \Z,\ w\in W. $$
In view of this, the sum in (\ref{unH-phi}) involves only integer powers of $x_1$,
so that applying $\text{Res}_{x_1}$ to the sum makes sense.
Second, due to the weak commutativity (\ref{weak-comm-phi(x)}), we see that for any $n\in \Z$,
the sum contains only finitely many nonzero terms, and
$u_{n}^{\mathcal{H}}\phi(x)=0$ for $n$ sufficiently large. Third, as $\phi(x)$ satisfies the condition (G1),
it can be readily seen that
$u_{n}^{\mathcal{H}}\phi(x)$ also satisfies the condition (G1).
\end{remark}

Now, set
 \begin{align}
 Y_{\mathcal{H}}(u,z)\phi(x)=\sum_{n\in \Z}u_{n}^{\mathcal{H}}\phi(x) z^{-n-1}.
 \end{align}
In terms of this generating function, we have
\begin{alignat}{1}
 & Y_{\mathcal{H}}(u,z)\phi(x)\nonumber \\
 =\  \  &\mbox{Res}_{x_{1}}\left(\frac{x+z}{x_1}\right)^{-\frac{r}{T}}\left(z^{-1}\delta\left(\frac{x_{1}-x}{z}\right)
 Y_{M_{3}}(u,x_{1})\phi(x)-z^{-1}\delta\left(\frac{x-x_{1}}{-z}\right)\phi(x)Y_{M_{2}}(u,x_{1})\right).\label{H action}
\end{alignat}
Using linearity, we obtain a vertex operator map
\begin{align}
Y_{\mathcal{H}}(\cdot,z):\  & V\to\left(\mbox{End}({\rm Hom}(M_{2},M_{3}))\{x\}\right)[[z,z^{-1}]]\nonumber\\
&u\mapsto Y_{\mathcal{H}}(u,z).
\end{align}

First, we present some technical results.

\begin{lemma}\label{technical-1}
Let $v\in V^r$ with $r\in \Z$ and let $\phi(x)\in \mathcal{H}(M_{2},M_{3})$. Then
\begin{align}\label{weak-assoc-new}
z^k(x+z)^{\frac{r}{T}}Y_{\mathcal{H}}(u,z)\phi(x)
=\left((x_1-x)^{k}x_1^{\frac{r}{T}}Y_{M_{3}}(u,x_{1})\phi(x)\right)|_{x_1=x+z},
\end{align}
where $k$ is any nonnegative integer such that (\ref{weak-comm-phi(x)}) holds.
Furthermore, we have
\begin{align}
 &x_1^{-1}\delta\left(\frac{x+z}{x_1}\right)\left(\frac{x+z}{x_1}\right)^{\frac{r}{T}}Y_{\mathcal{H}}(u,z)\phi(x)\nonumber \\
 =\  \  &z^{-1}\delta\left(\frac{x_{1}-x}{z}\right)
 Y_{M_{3}}(u,x_{1})\phi(x)-z^{-1}\delta\left(\frac{x-x_{1}}{-z}\right)\phi(x)Y_{M_{2}}(u,x_{1}).\label{H-Jacobi}
\end{align}
\end{lemma}

\begin{proof} Write $r=qT+r_0$ for some $q,r_0\in \Z$ with  $0\le r_0<T$.
As $Y_{\mathcal{H}}(u,z)\phi(x)$ involves only finitely many negative powers of $z$, we can multiply both sides of
(\ref{H action}) by $(x+z)^{\frac{r_0}{T}}$ and by doing so we get
\begin{align*}
 &z^k(x+z)^{\frac{r_0}{T}} Y_{\mathcal{H}}(u,z)\phi(x)\nonumber \\
 =\  \  &\mbox{Res}_{x_{1}}z^kx_1^{\frac{r_0}{T}}\left(z^{-1}\delta\left(\frac{x_{1}-x}{z}\right)
 Y_{M_{3}}(u,x_{1})\phi(x)-z^{-1}\delta\left(\frac{x-x_{1}}{-z}\right)\phi(x)Y_{M_{2}}(u,x_{1})\right)\nonumber\\
=\  \  &\mbox{Res}_{x_{1}} z^{-1}\delta\left(\frac{x_{1}-x}{z}\right) \left((x_1-x)^{k}x_1^{\frac{r_0}{T}}Y_{M_{3}}(u,x_{1})\phi(x)\right)\nonumber\\
&  -\mbox{Res}_{x_{1}}z^{-1}\delta\left(\frac{x-x_{1}}{-z}\right)\left((x_1-x)^{k}x_1^{\frac{r_0}{T}}\phi(x)Y_{M_{2}}(u,x_{1})\right)
\nonumber\\
=\  \ &\mbox{Res}_{x_{1}}x_1^{-1}\delta\left(\frac{x+z}{x_1}\right)
\left((x_1-x)^{k}x_1^{\frac{r_0}{T}}Y_{M_{3}}(u,x_{1})\phi(x)\right).
\end{align*}
Then
\begin{align*}
 &z^k(x+z)^{\frac{r}{T}} Y_{\mathcal{H}}(u,z)\phi(x)\nonumber \\
 =\  \  &z^k(x+z)^{q}(x+z)^{\frac{r_0}{T}} Y_{\mathcal{H}}(u,z)\phi(x)\nonumber \\
=\  \  &\mbox{Res}_{x_{1}}x_1^{-1}\delta\left(\frac{x+z}{x_1}\right)(x+z)^{q}
\left((x_1-x)^{k}x_1^{\frac{r_0}{T}}Y_{M_{3}}(u,x_{1})\phi(x)\right)\nonumber \\
=\  \  &\mbox{Res}_{x_{1}}x_1^{-1}\delta\left(\frac{x+z}{x_1}\right)x_1^{q}
\left((x_1-x)^{k}x_1^{\frac{r_0}{T}}Y_{M_{3}}(u,x_{1})\phi(x)\right)\nonumber \\
=\  \  &\mbox{Res}_{x_{1}}x_1^{-1}\delta\left(\frac{x+z}{x_1}\right)
\left((x_1-x)^{k}x_1^{\frac{r}{T}}Y_{M_{3}}(u,x_{1})\phi(x)\right)\nonumber \\
=\  \  & \left((x_1-x)^{k}x_1^{\frac{r}{T}}Y_{M_{3}}(u,x_{1})\phi(x)\right)|_{x_1=x+z}.
\end{align*}
This proves the first assertion.
Furthermore, we have
\begin{align*}
 &z^kx_1^{\frac{r}{T}}\left(z^{-1}\delta\left(\frac{x_{1}-x}{z}\right)
 Y_{M_{3}}(u,x_{1})\phi(x)-z^{-1}\delta\left(\frac{x-x_{1}}{-z}\right)\phi(x)Y_{M_{2}}(u,x_{1})\right)\nonumber\\
 =\  \  &z^{-1}\delta\left(\frac{x_{1}-x}{z}\right)\left(x_1^{\frac{r}{T}} (x_1-x)^{k}
 Y_{M_{3}}(u,x_{1})\phi(x)\right)-z^{-1}\delta\left(\frac{x-x_{1}}{-z}\right)\left(x_1^{\frac{r}{T}}(x_1-x)^{k}\phi(x)Y_{M_{2}}(u,x_{1})\right)
 \nonumber\\
 =\  \  &x_1^{-1}\delta\left(\frac{x+z}{x_1}\right)\left( (x_1-x)^{k}x_1^{\frac{r}{T}}
 Y_{M_{3}}(u,x_{1})\phi(x)\right)\nonumber\\
 =\  \  &x_1^{-1}\delta\left(\frac{x+z}{x_1}\right)
 \left( (x_1-x)^{k}x_1^{\frac{r}{T}}Y_{M_{3}}(u,x_{1})\phi(x)\right)|_{x_1=x+z}\nonumber\\
 =\  \  &x_1^{-1}\delta\left(\frac{x+z}{x_1}\right)z^{k}(x+z)^{\frac{r}{T}}Y_{\mathcal{H}}(u,z)\phi(x),
\end{align*}
which is equivalent to (\ref{H-Jacobi}).
\end{proof}

Note that for any $u\in V$, we have  $u=u^0+u^1+\cdots +u^{T-1}$, where for $0\le r\le T-1$,
$$u^r=\frac{1}{T}\sum_{j=0}^{T-1}\eta_{T}^{-jr}\tau^j(u)\in V^r.$$
Using this and Lemma \ref{technical-1}, we immediately get the
the following generalization:

\begin{proposition}\label{technical-1.2}
Let $u\in V$ and let $\phi(x)\in \mathcal{H}(M_{2},M_{3})$. Then
\begin{align}\label{weak-assoc-new}
z^kY_{\mathcal{H}}(u,z)\phi(x)
=\left((x_1-x)^{k}Y_{M_{3}}(u,x_{1})\phi(x)\right)|_{x_1^{1/T}=(x+z)^{1/T}},
\end{align}
where $k$ is any nonnegative integer such that (\ref{weak-comm-phi(x)}) holds.
Furthermore, we have
\begin{align}
 &\sum_{j=0}^{T-1}x_1^{-1}\delta\left(\eta_T^{-j}\left(\frac{x+z}{x_1}\right)^{\frac{1}{T}}\right)
 Y_{\mathcal{H}}(\tau^{j}u,z)\phi(x)\nonumber \\
 =\  \  &z^{-1}\delta\left(\frac{x_{1}-x}{z}\right)
 Y_{M_{3}}(u,x_{1})\phi(x)-z^{-1}\delta\left(\frac{x-x_{1}}{-z}\right)\phi(x)Y_{M_{2}}(u,x_{1}).\label{H-Jacobi-inhomg}
\end{align}
\end{proposition}

We also have the following generalization of the first part of Lemma \ref{technical-1}
(cf. \cite{Li-nonlocal}, Lemma 2.15):

\begin{lemma}\label{technical-2}
Let $v\in V$ and let $\phi_1(x),\dots,\phi_p(x)\in \mathcal{H}(M_{2},M_{3})$. If
$$\sum_{j=1}^{p}f_j(x_1-x) Y_{M_3}(v,x_1)\phi_j(x)=\sum_{j=1}^{p}f_j(x_1-x) \phi_j(x)Y_{M_2}(v,x_1),$$
where $f_1(x),\dots,f_p(x)\in \C[x]$, then
\begin{align}\label{sum-equality1}
\left(\sum_{j=1}^{p}f_j(x_1-x) Y_{M_3}(v,x_1)\phi_j(x)\right)|_{x_1^{1/T}=(x+z)^{1/T}}
=\sum_{j=1}^{p}f_j(z)Y_{\mathcal{H}}(v,z)\phi_j(x).
\end{align}
On the other hand, let $v^{1},\dots,v^{p}\in V$ and let $\phi(x)\in \mathcal{H}(M_{2},M_{3})$.
If
$$\sum_{j=1}^{p}f_j(x_1-x) Y_{M_3}(v^j,x_1)\phi(x)=\sum_{j=1}^{p}f_j(x_1-x) \phi(x)Y_{M_2}(v^j,x_1),$$
where $f_1(x),\dots,f_p(x)\in \C[x]$, then
\begin{align}
\left(\sum_{j=1}^{p}f_j(x_1-x)Y_{M_3}(v^j,x_1)\phi(x)\right)|_{x_1^{1/T}=(x+z)^{1/T}}
=\sum_{j=1}^{p}f_j(z)Y_{\mathcal{H}}(v^j,z)\phi(x).
\end{align}
\end{lemma}

\begin{proof} As $\phi_j(x)\in \mathcal{H}(M_{2},M_{3})$, there is a nonnegative integer $k$ such that
$$(x_1-x)^{k}Y_{M_3}(v,x_1)\phi_j(x)=(x_1-x)^{k}\phi_j(x)Y_{M_2}(v,x_1)$$
for all $1\le j\le p$. Then
$$(x_1-x)^{k}f_j(x_1-x)Y_{M_3}(v,x_1)\phi_j(x)=(x_1-x)^{k}f_j(x_1-x)\phi_j(x)Y_{M_2}(v,x_1).$$
In view of Lemma \ref{technical-1}, for $1\le j\le p$, we have
$$\left((x_1-x)^{k}Y_{M_3}(v,x_1)\phi_j(x)\right)|_{x_1^{1/T}=(x+z)^{1/T}}
=z^{k}Y_{\mathcal{H}}(v,z)\phi_j(x),$$
so that
\begin{align*}
&\left((x_1-x)^{k}f_j(x_1-x)Y_{M_3}(v,x_1)\phi_j(x)\right)|_{x_1^{1/T}=(x+z)^{1/T}}\nonumber\\
=\ &f_j(x_1-x)|_{x_1=x+z}\cdot \left((x_1-x)^{k}Y_{M_3}(v,x_1)\phi_j(x)\right)|_{x_1^{1/T}=(x+z)^{1/T}}\nonumber\\
=\ &z^{k}f_j(z) Y_{\mathcal{H}}(v,z)\phi_j(x).
\end{align*}
Then
\begin{align*}
&z^{k}\left(\sum_{j=1}^{p}f_j(x_1-x) Y_{M_3}(v,x_1)\phi_j(x)\right)|_{x_1^{1/T}=(x+z)^{1/T}}\nonumber\\
=\  & \left((x_1-x)^{k}\sum_{j=1}^{p}f_j(x_1-x) Y_{M_3}(v,x_1)\phi_j(x)\right)|_{x_1^{1/T}=(x+z)^{1/T}}\nonumber\\
=\  &\sum_{j=1}^{p} \left((x_1-x)^{k}f_j(x_1-x) Y_{M_3}(v,x_1)\phi_j(x)\right)|_{x_1^{1/T}=(x+z)^{1/T}}\nonumber\\
=\ z^k &\sum_{j=1}^{p}f_j(z)Y_{\mathcal{H}}(v,z)\phi_j(x),
\end{align*}
which amounts to (\ref{sum-equality1}). The second assertion can be proved similarly.
\end{proof}

As the first main result of this section, we have:

\begin{theorem} \label{H is a weak module}
Let $\tau\in {\rm Aut}(V)$ with $o(\tau)=T$ and let $M_{2}$ and $M_{3}$
be weak $\tau$-twisted $V$-modules. Then $(\mathcal{H}(M_{2},M_{3}),Y_{\mathcal{H}})$
carries the structure of a weak $V$-module.
\end{theorem}

\begin{proof} First, we prove that $\mathcal{H}(M_{2},M_{3})$
is closed under the action of $V$ given by (\ref{H action}), i.e.,
$$v_{n}^{\mathcal{H}}\phi(x)\in \mathcal{H}(M_{2},M_{3})\   \
\mbox{ for }v\in V,\ n\in \Z,\ \phi(x)\in \mathcal{H}(M_{2},M_{3}).$$
Assume $v\in V^{r}$ with $r\in \Z$. As it was pointed out before, $v_{n}^{\mathcal{H}}\phi(x)$ satisfies condition
(G1) for every $n\in \Z$.  We now establish condition (G2). Using (G2) for $\phi(x)$ and the corresponding property for
$Y_{M_{3}}(v,x_{1})$, we have
\begin{align*}
&[L(-1), Y_{M_{3}}(v,x_{1})\phi(x)]=\left(\frac{\partial}{\partial x_1}+\frac{\partial}{\partial x}\right)Y_{M_{3}}(v,x_{1})\phi(x),\\
&[L(-1),\phi(x)Y_{M_{2}}(v,x_{1})]=\left(\frac{\partial}{\partial x_1}+\frac{\partial}{\partial x}\right)\phi(x)Y_{M_{2}}(v,x_{1}).
\end{align*}
On the other hand, notice that
\begin{align*}
&\left(\frac{\partial}{\partial x_1}+\frac{\partial}{\partial x}\right)z^{-1}\delta\left(\frac{x_{1}-x}{z}\right)=0,\  \  \  \  \
\left(\frac{\partial}{\partial x_1}+\frac{\partial}{\partial x}\right)z^{-1}\delta\left(\frac{x-x_{1}}{-z}\right)=0,\\
&\left(\frac{\partial}{\partial x_1}+\frac{\partial}{\partial x}\right)\left(x_1^{-1}\delta\left(\frac{x+z}{x_1}\right)\left(\frac{x+z}{x_1}\right)^{\frac{r}{T}}\right)=0.
\end{align*}
Using these properties and (\ref{H-Jacobi}), we get
\begin{align}
&x_1^{-1}\delta\left(\frac{x+z}{x_1}\right)\left(\frac{x+z}{x_1}\right)^{\frac{r}{T}}[L(-1),Y_{\mathcal{H}}(v,z)\phi(x)]\nonumber \\
=\  \  &z^{-1}\delta\left(\frac{x_{1}-x}{z}\right)
 [L(-1),Y_{M_{3}}(v,x_{1})\phi(x)]-z^{-1}\delta\left(\frac{x-x_{1}}{-z}\right)[L(-1),\phi(x)Y_{M_{2}}(v,x_{1})]\nonumber\\
 =\ \ &\left(\frac{\partial}{\partial x_1}+\frac{\partial}{\partial x}\right)
 \left(z^{-1}\delta\left(\frac{x_{1}-x}{z}\right)
 Y_{M_{3}}(v,x_{1})\phi(x)-z^{-1}\delta\left(\frac{x-x_{1}}{-z}\right)\phi(x)Y_{M_{2}}(v,x_{1})\right)\nonumber\\
 =\ \ &\left(\frac{\partial}{\partial x_1}+\frac{\partial}{\partial x}\right)
 \left(x_1^{-1}\delta\left(\frac{x+z}{x_1}\right)\left(\frac{x+z}{x_1}\right)^{\frac{r}{T}}Y_{\mathcal{H}}(v,z)\phi(x)\right)
 \nonumber \\
 =\  \  &x_1^{-1}\delta\left(\frac{x+z}{x_1}\right)\left(\frac{x+z}{x_1}\right)^{\frac{r}{T}}
 \left(\frac{\partial}{\partial x_1}+\frac{\partial}{\partial x}\right)(Y_{\mathcal{H}}(v,z)\phi(x))
 \nonumber \\
=\  \  &x_1^{-1}\delta\left(\frac{x+z}{x_1}\right)\left(\frac{x+z}{x_1}\right)^{\frac{r}{T}}
 \frac{\partial}{\partial x}(Y_{\mathcal{H}}(v,z)\phi(x)).
\end{align}
As $Y_{\mathcal{H}}(v,z)\phi(x)$ involves only finitely many negative powers of $z$, by cancellation we obtain
$$[L(-1),Y_{\mathcal{H}}(v,z)\phi(x)]= \frac{\partial}{\partial x}Y_{\mathcal{H}}(v,z)\phi(x).$$
This implies that $v_{n}^{\mathcal{H}}\phi(x)$ satisfy condition
(G2) for all $n\in \Z$.

Next, we consider condition (G3). Let $u\in V$.
Then there exists a positive integer $k$ such that
\begin{align*}
&(x_{1}-x_{2})^{k}Y_{M_{i}}(u,x_{1})Y_{M_{i}}(v,x_{2})=(x_{1}-x_{2})^{k}Y_{M_{i}}(v,x_{2})Y_{M_{i}}(u,x_{1}),\\
&(x_{1}-x_{2})^{k}Y_{M_{3}}(u,x_{1})\phi(x_{2})=(x_{1}-x_{2})^{k}\phi(x_{2})Y_{M_{2}}(u,x_{1})
\end{align*}
for $i=2,3$.
Using these and (\ref{H-Jacobi}) we get
\begin{align*}
&x_1^{-1}\delta\left(\frac{x+z}{x_1}\right)\left(\frac{x+z}{x_1}\right)^{\frac{r}{T}}
(y-x)^{k}(y-x-z)^{k}Y_{M_3}(u,y)(Y_{\mathcal{H}}(v,z)\phi(x))\nonumber \\
=\  \ &x_1^{-1}\delta\left(\frac{x+z}{x_1}\right)\left(\frac{x+z}{x_1}\right)^{\frac{r}{T}}
(y-x)^{k}(y-x_1)^{k}Y_{M_3}(u,y)(Y_{\mathcal{H}}(v,z)\phi(x))\nonumber \\
=\ \ &x_1^{-1}\delta\left(\frac{x+z}{x_1}\right)\left(\frac{x+z}{x_1}\right)^{\frac{r}{T}}
(y-x)^{k}(y-x_1)^{k}(Y_{\mathcal{H}}(v,z)\phi(x))Y_{M_2}(u,y)\nonumber \\
=\ \ &x_1^{-1}\delta\left(\frac{x+z}{x_1}\right)\left(\frac{x+z}{x_1}\right)^{\frac{r}{T}}
(y-x)^{k}(y-x-z)^{k}(Y_{\mathcal{H}}(v,z)\phi(x))Y_{M_2}(u,y).
\end{align*}
Again, as $(Y_{\mathcal{H}}(v,z)\phi(x))$ involves only finitely many negative powers of $z$, by cancellation we obtain
\begin{align}\label{H-products}
(y-x)^{k}(y-x-z)^{k}Y_{M_3}(u,y)(Y_{\mathcal{H}}(v,z)\phi(x))
=(y-x)^{k}(y-x-z)^{k}(Y_{\mathcal{H}}(v,z)\phi(x))Y_{M_2}(u,y).
\end{align}
Let $n_0\in \Z$ such that $v^{\mathcal{H}}_{n}\phi(x)=0$ for all $n\ge n_0$.
For $m\in \N$, applying $\text{Res}_{z}z^{n_0-m}$ to both sides we get
\begin{align}\label{recursion}
&\sum_{j=0}^{k}\binom{k}{j}(-1)^{j}(y-x)^{2k-j}Y_{M_3}(u,y)\left(v_{n_0-m+j}^{\mathcal{H}}\phi(x)\right)\nonumber\\
=\ &\sum_{j=0}^{k}\binom{k}{j}(-1)^{j}(y-x)^{2k-j}\left(v_{n_0-m+j}^{\mathcal{H}}\phi(x)\right)Y_{M_3}(u,y).
\end{align}
It then follows from this identity and induction on $m$ that for every  $m\in \N$, there exists a positive integer $p$ such that
$$(y-x)^{p}Y_{M_3}(u,y)\left(v_{n_0-m}^{\mathcal{H}}\phi(x)\right)=(y-x)^{p}\left(v_{n_0-m}^{\mathcal{H}}\phi(x)\right)Y_{M_3}(u,y).$$
(Note that if $k=0$, it is clear, and if $k\ge 1$, (\ref{recursion}) gives a recursion formula.)
Thus we have proved $v_{n}^{\mathcal{H}}\phi(x)\in \mathcal{H}(M_{2},M_{3})$ for $n\in \Z$.
Now, we have proved that $Y_{\mathcal{H}}(\cdot,z)$ is a linear map from $V$ to
$\left(\mbox{End}({\mathcal{H}}(M_{2},M_{3}))\right)[[z,z^{-1}]]$
such that $Y_{\mathcal{H}}(u,z)\phi(x)\in {\mathcal{H}}(M_{2},M_{3})((z))$ for $u\in V,\ \phi(x)\in {\mathcal{H}}(M_{2},M_{3})$.
It can be readily seen from (\ref{H action})  that $Y_{\mathcal{H}}({\bf 1},z)=1$.

To prove that $\mathcal{H}(M_{2},M_{3})$
is a weak $V$-module, it suffices to prove the weak associativity.
Let $u\in V^r,\ v\in V^s$ with $r,s\in \Z$ and let $\phi(x)\in {\mathcal{H}}(M_2,M_3)$.
Then there exists a positive integer $n$ such that
\begin{align*}
&(x_1-x)^{n}Y_{M_3}(u,x_1)\phi(x)=(x_1-x)^{n}\phi(x)Y_{M_2}(u,x_1),\\
&(x_1-x)^{n}Y_{M_3}(v,x_1)\phi(x)=(x_1-x)^{n}\phi(x)Y_{M_2}(v,x_1),\\
&(x_1-x)^{n}Y_{M_3}(u,x_1)Y_{M_3}(v,x)=(x_1-x)^{n}Y_{M_3}(v,x)Y_{M_3}(u,x_1).
\end{align*}
From the twisted Jacobi identity for $(M_3,Y_{M_3})$ we have
\begin{align*}
&(y_2+z_0)^{\frac{r}{T}}Y_{M_3}(Y(u,z_0)v,y_2)\nonumber\\
=\ & \text{Res}_{y_1}y_{1}^{\frac{r}{T}}\left(x_0^{-1}\delta\left(\frac{y_1-y_2}{z_0}\right)Y_{M_3}(u,y_1)Y_{M_3}(v,y_2)
-x_0^{-1}\delta\left(\frac{y_2-y_1}{-z_0}\right)Y_{M_3}(v,y_2)Y_{M_3}(u,y_1)\right).
\end{align*}
Using these properties and the standard delta-function substitutions, we get
\begin{align*}
&(y_2+z_0)^{\frac{r}{T}}(y_2-x)^{n}(y_2-x+z_0)^{n}Y_{M_3}(Y(u,z_0)v,y_2)\phi(x)\nonumber\\
=\ &(y_2+z_0)^{\frac{r}{T}} (y_2-x)^{n}(y_2-x+z_0)^{n}\phi(x)Y_{M_2}(Y(u,z_0)v,y_2).
\end{align*}
Multiplying both sides by $(y_2+z_0)^{-\frac{r}{T}}$ (noticing that we are allowed to do so since $Y(u,z_0)v\in V((z_0))$),
we obtain
\begin{align}
&(y_2-x)^{n}(y_2-x+z_0)^{n}Y_{M_3}(Y(u,z_0)v,y_2)\phi(x)\nonumber\\
=\ &(y_2-x)^{n}(y_2-x+z_0)^{n}\phi(x)Y_{M_2}(Y(u,z_0)v,y_2).
\end{align}
Noticing that $u_{m}v\in V^{r+s}$ for $m\in \Z$, using Lemma \ref{technical-2} we obtain
\begin{align}
&z_2^{n}(z_2+z_0)^{n}(x+z_2)^{\frac{r+s}{T}}Y_{\mathcal{H}}(Y(u,z_0)v,z_2)\phi(x)\nonumber\\
=\  &\left((x_2-x)^{n}(x_2-x+z_0)^{n}x_2^{\frac{r+s}{T}}Y_{M_3}(Y(u,z_0)v,x_2)\phi(x)\right)|_{x_2=x+z_2}.
\end{align}
On the other hand, we have
\begin{align*}
(x_2+z_0)^{\frac{r}{T}}z_0^{n}Y_{M_3}(Y(u,z_0)v,x_2)
=\left((x_1-x_2)^{n}x_1^{\frac{r}{T}}Y_{M_3}(u,x_1)Y_{M_3}(v,x_2)\right)|_{x_1=x_2+z_0},
\end{align*}
which gives
\begin{align*}
z_0^{n}Y_{M_3}(Y(u,z_0)v,x_2)
=(x_2+z_0)^{-\frac{r}{T}}\left((x_1-x_2)^{n}x_1^{\frac{r}{T}}Y_{M_3}(u,x_1)Y_{M_3}(v,x_2)\right)|_{x_1=x_2+z_0}.
\end{align*}
Then
\begin{align*}
&z_2^{n}z_0^{n}(z_2+z_0)^{n}(x+z_2)^{\frac{r+s}{T}}Y_{\mathcal{H}}(Y(u,z_0)v,z_2)\phi(x)\nonumber\\
=\  &\left((x_2-x)^{n}(x_2-x+z_0)^{n}x_2^{\frac{r+s}{T}}z_0^{n}Y_{M_3}(Y(u,z_0)v,x_2)\phi(x)\right)|_{x_2=x+z_2}\nonumber\\
=\ &\left((x_2-x)^{n}(x_2-x+z_0)^{n}x_2^{\frac{r+s}{T}}(x_2+z_0)^{-\frac{r}{T}} \left((x_1-x_2)^{n}x_1^{\frac{r}{T}}
Y_{M_3}(u,x_1)Y_{M_3}(v,x_2)\phi(x)\right)|_{x_1=x_2+z_0}\right)|_{x_2=x+z_2}\nonumber\\
=\ &\left((1+z_0/x_2)^{-\frac{r}{T}} \left((x_2-x)^{n}(x_1-x)^{n}(x_1-x_2)^{n}x_1^{\frac{r}{T}}
x_2^{\frac{s}{T}}Y_{M_3}(u,x_1)Y_{M_3}(v,x_2)\phi(x)\right)|_{x_1=x_2+z_0}\right)|_{x_2=x+z_2}\nonumber\\
=\ &\left(1+\frac{z_0}{x+z_2}\right)^{-\frac{r}{T}} \left(\left((x_2-x)^{n}(x_1-x)^{n}(x_1-x_2)^{n}x_1^{\frac{r}{T}}
x_2^{\frac{s}{T}}Y_{M_3}(u,x_1)Y_{M_3}(v,x_2)\phi(x)\right)|_{x_1=x_2+z_0}\right)|_{x_2=x+z_2},
\end{align*}
which gives
\begin{align}\label{iterate-1}
&z_2^{n}z_0^{n}(z_2+z_0)^{n}(x+z_2)^{\frac{s}{T}}(x+z_2+z_0)^{\frac{r}{T}} Y_{\mathcal{H}}(Y(u,z_0)v,z_2)\phi(x)\nonumber\\
=\ &\left(\left((x_2-x)^{n}(x_1-x)^{n}(x_1-x_2)^{n}x_1^{\frac{r}{T}}
x_2^{\frac{s}{T}}Y_{M_3}(u,x_1)Y_{M_3}(v,x_2)\phi(x)\right)|_{x_1=x_2+z_0}\right)|_{x_2=x+z_2}.
\end{align}

On the other hand, we consider $Y_{\mathcal{H}}(u,z_1)Y_{\mathcal{H}}(v,z_2)\phi(x)$.
With (\ref{H-products}), by Lemma \ref{technical-2} we have
\begin{align*}
&z_1^{k}(z_1-z_2)^{k}(x+z_1)^{\frac{r}{T}}Y_{\mathcal{H}}(u,z_1)Y_{\mathcal{H}}(v,z_2)\phi(x)\nonumber\\
=\ &\left((x_1-x)^{k}(x_1-x-z_2)^{k}x_1^{\frac{r}{T}}Y_{M_3}(u,x_1)(Y_{\mathcal{H}}(v,z_2)\phi(x))\right)|_{x_1=x+z_1}.
\end{align*}
We also have
$$z_2^{k}(x+z_2)^{\frac{s}{T}}Y_{\mathcal{H}}(v,z_2)\phi(x)
=\left( (x_2-x)^{k}x_2^{\frac{s}{T}}Y_{M_3}(v,x_2)\phi(x)\right)|_{x_2=x+z_2}.$$
Then we get
\begin{align*}
&z_2^{k}(x+z_2)^{\frac{s}{T}}z_1^{k}(z_1-z_2)^{k}(x+z_1)^{\frac{r}{T}}Y_{\mathcal{H}}(u,z_1)Y_{\mathcal{H}}(v,z_2)\phi(x)\nonumber\\
=\ &\left((x_1-x)^{k}(x_1-x-z_2)^{k}x_1^{\frac{r}{T}}Y_{M_3}(u,x_1)
\left( (x_2-x)^{k}x_2^{\frac{s}{T}}Y_{M_3}(v,x_2)\phi(x)\right)|_{x_2=x+z_2}\right)|_{x_1=x+z_1}\nonumber\\
=\ &\left( \left((x_1-x)^{k}(x_1-x_2)^{k}
(x_2-x)^{k}x_1^{\frac{r}{T}}x_2^{\frac{s}{T}}Y_{M_3}(u,x_1)Y_{M_3}(v,x_2)\phi(x)\right)|_{x_2=x+z_2}\right)|_{x_1=x+z_1},
\end{align*}
which gives
\begin{align*}
&z_2^{k}(x+z_2)^{\frac{s}{T}}(z_2+z_0)^{k}z_0^{k}(x+z_2+z_0)^{\frac{r}{T}}Y_{\mathcal{H}}(u,z_2+z_0)Y_{\mathcal{H}}(v,z_2)\phi(x)\nonumber\\
=\ &\left( \left((x_1-x)^{k}(x_1-x_2)^{k}
(x_2-x)^{k}x_1^{\frac{r}{T}}x_2^{\frac{s}{T}}Y_{M_3}(u,x_1)Y_{M_3}(v,x_2)\phi(x)\right)|_{x_2=x+z_2}\right)|_{x_1=x+z_2+z_0}.
\end{align*}
Combining this with (\ref{iterate-1}) (replacing $n$ with a larger one so that $n\ge k$) we obtain
\begin{align*}
&z_2^{n}z_0^{n}(z_2+z_0)^{n}(x+z_2)^{\frac{s}{T}}(x+z_2+z_0)^{\frac{r}{T}} Y_{\mathcal{H}}(Y(u,z_0)v,z_2)\phi(x)\nonumber\\
=&\  z_2^{n}(x+z_2)^{\frac{s}{T}}(z_2+z_0)^{n}z_0^{n}(x+z_2+z_0)^{\frac{r}{T}}Y_{\mathcal{H}}(u,z_2+z_0)Y_{\mathcal{H}}(v,z_2)\phi(x),
\end{align*}
noticing that under the formal series expansion convention, for any $m\in \Z$,
$$(x_2+z_0)^{m}|_{x_2=x+z_2}=((x+z_2)+z_0)^{m}=(x+(z_2+z_0))^{m}=x_1^{m}|_{x_1=x+z_2+z_0}.$$
Multiplying both sides by $z_2^{-n}z_0^{-n}(x+z_2)^{-\frac{s}{T}}(x+z_2+z_0)^{-\frac{r}{T}}$ we get
\begin{align*}
(z_2+z_0)^{n}Y_{\mathcal{H}}(Y(u,z_0)v,z_2)\phi(x)
=(z_2+z_0)^{n}Y_{\mathcal{H}}(u,z_2+z_0)Y_{\mathcal{H}}(v,z_2)\phi(x).
\end{align*}
This establishes the weak associativity. Therefore, $(\mathcal{H}(M_{2},M_{3}),Y_{\mathcal{H}})$
carries the structure of a weak $V$-module.
\end{proof}

We end this section with the following technical result:

\begin{lemma}\label{generating-H}
Let $\tau,\ T,\  M_{2},\  M_{3}$ be given as before and let $\phi(x)\in (\text{Hom}(M_2,M_3))\{z\}$.
Suppose that $\phi(x)$ satisfies the conditions (G1) and (G2) and suppose that the condition (G3) holds for any $u\in U$, where
$U$ is a subset of $V$ such that $\cup_{j=0}^{k-1}\tau^{j}(U)$ generates  $V$ as a vertex algebra.
Then $\phi(x)\in {\mathcal{H}}(M_2,M_3)$.
\end{lemma}

\begin{proof} Let $K$ consist of all vectors $a\in V$ such that
$$(x_1-x)^{n}Y_{M_3}(a,x_1)\phi(x)=(x_1-x)^{n}\phi(x)Y_{M_2}(a,x_1)$$
for some nonnegative integer $n$. We must prove $K=V$.
As $\cup_{j=0}^{T-1}\tau^{j}(U)$ generates $V$ as a vertex algebra,  it suffices to prove
that $K$ is a vertex subalgebra that contains $\cup_{j=0}^{T-1}\tau^{j}(U)$.
 It is clear that $K$ is a subspace with ${\bf 1}\in K$, and  from assumption we have $U\subset K$.
 Note that for any $v\in V$, we have
 $$Y_{M_i}(\tau(v),z)=\lim_{z^{1/T}\rightarrow \eta_T^{-1}z^{1/T}}Y_{M_i}(v,z)$$
 with $i=2,3$. It then follows that $\tau (K)\subset K$. Thus $\cup_{j=0}^{T-1}\tau^{j}(U)\subset K$.
Now, it remains to prove that $u_{m}v\in K$ for any $u,v\in K,\ m\in \Z$. As $\tau(K)\subset K$,
it suffices to consider $u\in K\cap V^r$ with $r\in \Z$. Note that
\begin{align}\label{twisted-module-iterate-formula}
&(z_2+z_0)^{\frac{r}{T}}Y_{M_i}(Y(u,z_0)v,z_2)\nonumber\\
=\ &\text{Res}_{z_1}z_1^{\frac{r}{T}}
\left(z_0^{-1}\delta\left(\frac{z_1-z_2}{z_0}\right)Y_{M_i}(u,z_1)Y_{M_i}(v,z_2)
-z_0^{-1}\delta\left(\frac{z_2-z_1}{-z_0}\right)Y_{M_i}(v,z_2)Y_{M_i}(u,z_1) \right)
\end{align}
holds on $M_i$ for $i=2,3$. Let $k$ be a positive integer such that
$$(z_1-z)^{k}Y_{M_3}(a,z_1)\phi(z)=(z_1-z)^{k}\phi(z)Y_{M_2}(a,z_1)$$
for $a\in \{u,v\}$. Using these relations and (\ref{twisted-module-iterate-formula}),
we get
$$(z_2-z)^{k}(z_2+z_0-z)^{k}Y_{M_3}(Y(u,z_0)v,z_2)\phi(z)=(z_2-z)^{k}(z_2+z_0-z)^{k}\phi(z)Y_{M_2}(Y(u,z_0)v,z_2).$$
Just as in the first part of the proof of Theorem \ref{H is a weak module},
it follows from this relation and an induction that $u_{m}v\in K$ for all $m\in \Z$. Therefore, we conclude $K=V$ and hence
$\phi(x)\in {\mathcal{H}}(M_2,M_3)$.
\end{proof}

\begin{proposition}\label{isomorphism of V-hom and Intertwining}
Let $M_{2},$ $M_{3}$ be weak $\tau$-twisted $V$-modules and let $M$ be any weak $V$-module.
Then for any intertwining operator $I(\cdot,x)$ of type
$\binom{M_3}{M\ M_{2}},$ we have $I(w,x)\in \mathcal{H}(M_{2},M_{3})$ for $w\in M$, and
the map $f_I$ defined by $f_I(w)=I(w,x)$ for $w\in M$ is a homomorphism of weak $V$-modules.
Furthermore, the linear map
$\theta: I_{V}\binom{M_3}{M\ M_{2}}\rightarrow \text{Hom}_{V}(M,\mathcal{H}(M_{2},M_{3}))$
defined by $\theta (I)=f_I$ is a linear isomorphism.
\end{proposition}

\begin{proof} Let  $I(\cdot,x)$ be any intertwining operator
of type $\binom{M_3}{M\ M_{2}}$. From definition, we see that
 for every $w\in M$, $I(w,x)$ satisfies conditions (G1) and (G2).
 As the Jacobi identity for $I(\cdot,x)$ implies commutator formula
 which furthermore implies locality, $I(w,x)$ also satisfies (G3).
Thus $I(w,x)\in\mathcal{H}(M_{2},M_{3})$.
Then we have a linear map $f_{I}$ from $M$ to
$\mathcal{H}(M_{2},M_{3})$ defined by $f_{I}(w)=I(w,x)$ for $w\in M$.
For any $u\in V^{r}$, $w\in M$ with $r\in \Z$, we have
\begin{alignat*}{1}
 & f_{I}(Y(u,x_{0})w)\\
 =\  & I(Y(u,x_{0})w,x)\\
 =\  & \mbox{Res}_{x_{1}}\left(\frac{x+x_{0}}{x_1}\right)^{-\frac{r}{T}}\left(x_{0}^{-1}\delta\left(\frac{x_{1}-x}{x_{0}}\right)Y_{M_{3}}(u,x_{1})I(w,x)-x_{0}^{-1}\delta\left(\frac{x-x_{1}}{-x_{0}}\right)I(w,x)Y_{M_{2}}(u,x_{1})\right)\\
 =\  & Y_{\mathcal{H}}(u,x_{0})I(w,x)\\
 =\  & Y_{\mathcal{H}}(u,x_{0})f_{I}(w).
\end{alignat*}
This proves that  $f_{I}$ is a homomorphism of weak $V$-modules.

For the second assertion, it is clear that $\theta$ is a one-to-one linear map.
To prove that it is onto, let $f$ be a $V$-homomorphism from $M$ to $\mathcal{H}(M_{2},M_{3}).$
For every $w\in M$, noticing that as an element of $\mathcal{H}(M_{2},M_{3})$, $f(w)$ is a formal series in a formal variable,
we write $f(w)$ as $f(w)(x)$, and then set $I_{f}(w,x)=f(w)(x)\in \mathcal{H}(M_{2},M_{3})$.
This gives a linear map $I_{f}(\cdot,x)$ from $M$ to $({\rm Hom}_{\C}(M_2,M_3))\{x\}$.
We have
$$[L(-1), I_{f}(w,z)]=\frac{d}{dz}I_{f}(w,z).$$
For any $v\in V,\ w\in M$, as $f(w)(x)\in {\mathcal{H}}(M_2,M_3)$, by condition (G3)
there exists a nonnegative integer $k$ such that
$$(x_1-x)^{k}Y_{M_3}(v,x_1)f(w)(x)=(x_1-x)^{k}f(w)(x)Y_{M_2}(v,x_1),$$
which amounts to
\begin{align}\label{weak-comm-If}
(x_1-x)^{k}Y_{M_3}(v,x_1)I_{f}(w,x)=(x_1-x)^{k}I_{f}(w,x)Y_{M_2}(v,x_1).
\end{align}
On the other hand, for $v\in V^r$ with $r\in\Z$ and $w\in M$, with $f$ a homomorphism of weak $V$-modules, we have
\begin{align*}
&(z_2+z_0)^{\frac{r}{T}}I_f(Y_M(v,z_0)w,z_2)\nonumber\\
=\ &(z_2+z_0)^{\frac{r}{T}} f(Y_{M}(v,z_0)w)(z_2)\nonumber\\
=\ &(z_2+z_0)^{\frac{r}{T}} Y_{\mathcal{H}}(v,z_0)f(w)(z_2)\nonumber\\
=\ & \text{Res}_{z_1}z_1^{\frac{r}{T}}\left(z_0^{-1}\delta\left(\frac{z_{1}-z_2}{z_0}\right)
 Y_{M_{3}}(v,z_{1})f(w)(z_2)-z_0^{-1}\delta\left(\frac{z_2-z_{1}}{-z_0}\right)f(w)(z_2)Y_{M_{2}}(v,z_{1})\right)\nonumber\\
=\ & \text{Res}_{z_1}z_1^{\frac{r}{T}}\left(z_0^{-1}\delta\left(\frac{z_{1}-z_2}{z_0}\right)
 Y_{M_{3}}(v,z_{1})I_{f}(w,z_2)-z_0^{-1}\delta\left(\frac{z_2-z_{1}}{-z_0}\right)I_{f}(w,z_2)Y_{M_{2}}(v,z_{1})\right).
\end{align*}
This together with (\ref{weak-comm-If}) gives the Jacobi identity. Therefore, $I_{f}(\cdot,x)$ is an intertwining operator
of type $\binom{M_3}{M\ M_{2}}.$ It is clear that $\theta(I_{f})=f$. This proves that $\theta$ is onto, concluding the proof.
\end{proof}

\section{Permutation orbifolds}
In this section, we shall apply the results of Section \ref{subsec:Space-of-generalized-general results}
to the permutation orbifold model.

\subsection{Linear isomorphism between $I_{V}\binom{W}{M\ N}$
and $I_{V^{\otimes k}}\binom{T_{\sigma}(W)}{M^{1}\ T_{\sigma}(N)}$ }

Let $V$ be a vertex operator algebra and let $k$ be a positive integer with $k\ge 2$, which are all fixed throughout this section.
Set $\sigma=\left(1\ 2\ \cdots k\right)$, which is an automorphism of $V^{\otimes k}$.
Let $M$, $N$, and $W$ be irreducible $V$-modules. Set $M^{1}=M\otimes V^{\otimes (k-1)},$
which is a $V^{\otimes k}$-module.
Let $T_{\sigma}(N)$ and $T_{\sigma}(W)$ be the $\sigma$-twisted
$V^{\otimes k}$-modules constructed in \cite{BDM}. In this section,
we prove that there is a canonical linear isomorphism between the spaces $I_{V}\binom{W} {M\ N}$
and $I_{V^{\otimes k}}\binom{T_{\sigma}(W)}{M^{1}\ T_{\sigma}(N)}$. That
is, given an intertwining operator of type $\binom{W}{M\ N}$,
we can construct an intertwining operator of type $\binom{T_{\sigma}(W)}{M^{1}\ T_{\sigma}(N)}$.
Conversely, given an intertwining operator of type $\binom{T_{\sigma}(W)}{M^{1}\ T_{\sigma}(N)}$,
we can construct an intertwining operator of type $\binom{W}{M\ N}$.
As one of the main results, we obtain the following tensor product relation:
\[
\left(M\otimes V^{\otimes (k-1)}\right)\boxtimes_{V^{\otimes k}}T_{\sigma}\left(N\right)
\simeq T_{\sigma}\left(M\boxtimes_{V}N\right).
\]

\subsubsection{From $I_{V}\binom{W}{M\ N}$ to
$I_{V^{\otimes k}}\binom{T_{\sigma}(W)}{M^1\ T_{\sigma}(N)}$ }

For any $u\in V$, $1\le j\le k$, we denote by $u^{j}$ the vector in $V^{\otimes k}$
whose $j^{\text{th }}$ tensor factor is $u$ and the other tensor
factors are $\boldsymbol{1}$:
\begin{align}
u^j={\bf 1}^{\otimes(j-1)}\otimes u\otimes {\bf 1}^{\otimes (k-j)}.
\end{align}
We have
\begin{align}
Y(u^j,z)=1^{\otimes(j-1)}\otimes Y(u,z)\otimes 1^{\otimes (k-j)}.
\end{align}
Note that $\sigma u^{j}=u^{j+1}$ for $1\le j\le k$,
where $u^{k+1}=u^1$ by convention.
For any weak $V^{\otimes k}$-module $(W,Y_W)$, we have
\[
\left[Y_W(u^{i},z_{1}),Y_W(v^{j},z_{2})\right]
=\text{Res}_{z_{0}}z_{2}^{-1}\delta\left(\frac{z_{1}-z_{0}}{z_{2}}\right)
Y_W(Y(u^{i},z_{0})v^{j},z_{2})
\]
for  $u,v\in V,\ 1\le i,j\le k$.
For $n\ge 0$, as $u_{n}{\bf 1} =0$ (in $V$)  we have $(u^i)_{n}(v^j)=0$ (in $V^{\otimes k}$) if $i\ne j$.
Then we immediately have:

\begin{lemma} \label{commutativity of bracket of u^i v^j}
 Let $(W,Y_W)$ be any weak $V^{\otimes k}$-module. Then
$\left[Y_W(u^{i},z_{1}),Y_W(v^{j},z_{2})\right]=0$
 for any $u,v\in V$, $i,j\in\left\{ 1,\dots,k\right\} $ with $i\not=j$.
\end{lemma}

Let $k$ be a positive integer. Recall from \cite{BDM}
\begin{align}\label{delta-kx}
\Delta_{k}(z)=\exp\left(\sum_{n\ge 1}a_{n}z^{-\frac{n}{k}}L(n)\right)k^{-L(0)}z^{(1/k-1)L(0)},
\end{align}
where the coefficients $a_n$ for $n\ge 1$ are uniquely determined by
\begin{align}\label{adefine}
\exp \left(\sum_{n\ge 1}-a_{n}x^{n+1}\frac{d}{dx}\right) x=\frac{1}{k}(1+x)^{k}-\frac{1}{k}.
\end{align}
As in \cite{BDM}, we shall also use the expression $\Delta_k(z^k)^{-1}$. For convenience,
we set
\begin{align}
\Phi_k(z)=\Delta_k(z^k)^{-1}=(kz^{k-1})^{L(0)}\exp\left(-\sum_{n\ge 1}a_{n}z^{-n}L(n)\right).
\end{align}

The following results were obtained in \cite{BDM} (Proposition 2.2 and Corollary 2.5):

\begin{proposition}\label{BDM-formulas}
 Let $(W,Y_{W})$ be any $V$-module. Then
\begin{align}
&\Delta_{k}(x)Y_{W}(v,z)w
=Y_{W}\left(\Delta_{k}(x+z)v,(x+z)^{1/k}-x^{1/k}\right)\Delta_{k}(x)w,\label{delta-conjugation}\\
&\Phi_{k}(x)Y_{W}(v,z)w
=Y_{W}\left(\Phi_{k}(x+z)v,(x+z)^{k}-x^{k}\right)\Phi_{k}(x)w,\\
&\Delta_{k}(z)L(-1)w-\frac{1}{k}z^{1/k-1}L(-1)\Delta_k(z)w=
\frac{d}{dz}\Delta_{k}(z)w,\\
&\Delta_{k}(z)\omega=\frac{1}{k^2}z^{2(\frac{1}{k}-1)}\omega+z^{-2}\frac{c}{24}(1-k^{-2}){\bf 1}
\end{align}
for $v\in V,\ w\in W$, where $\omega$ is the conformal vector of $V$ and $c$ is the central charge.
\end{proposition}

The following is one of the main results of \cite{BDM}:

\begin{theorem}\label{BDM}
Let $V$ be any vertex operator algebra and let $\sigma=(12\cdots k)\in \text{Aut}(V^{\otimes k})$.
Then for any (weak, admissible) $V$-module $(W,Y_W)$,  we have a (weak, admissible) $\sigma$-twisted
$V^{\otimes k}$-module $(T_{\sigma}(W),Y_{T_{\sigma}(W)})$, where $T_{\sigma}(W)=W$ as a vector space
and the vertex operator map $Y_{T_{\sigma}(W)}(\cdot,z)$ is uniquely determined by
\begin{align}
Y_{T_{\sigma}(W)}(u^1,z)=Y_{W}(\Delta_k(z)u,z^{1/k})\   \   \   \   \mbox{ for }u\in V,
\end{align}
where $u^1=u\otimes {\bf 1}^{\otimes (k-1)}\in V^{\otimes k}$.
Furthermore, every (weak, admissible) $\sigma$-twisted $V^{\otimes k}$-module is isomorphic to one in  this form.
\end{theorem}

Set $\eta_{k}=e^{\frac{2\pi i}{k}}$, the principal primitive $k$-th root of unity.

We shall need the $L(-1)$-operator for the vertex operator algebra $V^{\otimes k}$ on the $\sigma$-twisted
$V^{\otimes k}$-module $T_{\sigma}(W)$.
Note that the conformal vector of $V^{\otimes k}$ is
$$\omega^{(k)}:=\omega^1+\cdots +\omega^{k},$$
where $\omega$ denotes the conformal vector of $V$.
Write
\begin{align}
Y_{W}(\omega,z)=\sum_{n\in \Z}L(n)z^{-n-2},\   \   \   \  \
Y_{T_{\sigma}(W)}(\omega^{(k)},z)=\sum_{n\in \Z}L_{T_{\sigma}}(n)z^{-n-2}.
\end{align}
Using the covariance property of $\sigma$-twisted modules and the expression of $\Delta_k(z)\omega$ from
Proposition \ref{BDM-formulas}, we have
\begin{align*}
&Y_{T_{\sigma}(W)}(\omega^{(k)},z)=\sum_{j=0}^{k-1}Y_{T_{\sigma}(W)}(\sigma^{j}\omega^{1},z)
=\sum_{j=0}^{k-1}\lim_{z^{1/k}\rightarrow \eta_k^{-j}z^{1/k}}Y_{T_{\sigma}(W)}(\omega^{1},z)\nonumber\\
=&\  \sum_{j=0}^{k-1}\lim_{z^{1/k}\rightarrow \eta_k^{-j}z^{1/k}}Y_{W}(\Delta_k(z)\omega,z^{1/k})\nonumber\\
=&\ \sum_{j=0}^{k-1}\frac{1}{k^2}\eta_k^{-2j}z^{2/k-2}Y_{W}(\omega,\eta_k^{-j}z^{1/k})+z^{-2}\frac{c}{24}k(1-k^{-2})
\nonumber\\
=&\  \frac{1}{k}\sum_{m\in \Z}L(mk)z^{-m-2}  +z^{-2}\frac{c}{24}(k-k^{-1}).
\end{align*}
Then
\begin{align}
L_{T_{\sigma}}(n)=\begin{cases}\frac{1}{k}L(nk)&\   \text{ for }n\ne 0\\
\frac{1}{k}L(0)+\frac{c}{24}(k-k^{-1})&\   \text{ for }n= 0.
\end{cases}
\end{align}
In particular, we have
\begin{align}
L_{T_{\sigma}}(-1)=\frac{1}{k}L(-k).
\end{align}

We have the following result which will be used later:

\begin{lemma}\label{newlemma}
Let $(W,Y_{W})$ be any $V$-module.  Then
\begin{align}
\frac{d}{dz}\Delta_{k}(z)=\frac{1}{k}\sum_{i\ge 1 }{-k+1\choose i}z^{-1-(i-1)/k}L(i-1)\Delta_{k}(z)
\end{align}
holds on $W.$
\end{lemma}

\begin{proof} It is similar to that of Corollary 2.5 of \cite{BDM}.  As in \cite{BDM}
we set
\[\Delta_k^x (z) = \exp \Biggl( - \sum_{j \ge 1} a_j z^{- j/k} \Lx
\Biggr) k^{\Lo} z^{\left( - 1/k + 1 \right) \Lo}, \]
an element of $(\End \; \C[x, x^{-1}])[[z^{1/k}, z^{-1/k}]]$.
As in the proof of Corollary 2.5 of \cite{BDM}, it suffices to show that
\begin{equation}\label{newequation}
\frac{1}{k}\sum_{i\ge 1 }{-k+1\choose i}\left(-x^i\frac{\partial}{\partial x}\right)z^{-1-(i-1)/k}\Delta_{k}^x(z)=\frac{\partial}{\partial z}\Delta_{k}^x(z)
\end{equation}
in  $(\End \; \C[x, x^{-1}])
[[z^{1/k}, z^{-1/k}]].$

Note that
\begin{eqnarray*}
\Delta_k^x (z)x&=&\exp \Biggl( - \sum_{j\ge 1} a_j z^{- j/k} \Lx
\Biggr) k^{\Lo} z^{\left( - 1/k + 1 \right) \Lo} x\\
&=&\exp \Biggl( - \sum_{j\ge } a_j z^{- j/k} \Lx
\Biggr) k z^{ - 1/k + 1} x\\
&=&kz\exp \Biggl( - \sum_{j\ge 1} a_j z^{- j/k} \Lx
\Biggr)  z^{ - 1/k} x\\
&=&kz\exp \Biggl( - \sum_{j\ge 1} a_j (z^{- 1/k}x)^{j+1} \frac{\partial}{\partial z^{-1/k}x}
\Biggr) z^{ - 1/k} x\\
&=&kz\left(\frac{1}{k}(1+z^{-1/k}x)^k-\frac{1}{k}\right)\\
&=&z\left((1+z^{-1/k}x)^k-1\right),
\end{eqnarray*}
where for the second last equality we are using equation (\ref{adefine}).
Then
\begin{eqnarray*}
&&\frac{1}{k}\sum_{i\ge 1 }{-k+1\choose i}\left(-x^i\frac{\partial}{\partial x}\right)z^{-1-(i-1)/k}\Delta_{k}^x(z)x\\
&=&\frac{1}{k}\sum_{i\ge 1 }{-k+1\choose i}\left(-x^i\frac{\partial}{\partial x}\right)z^{-1-(i-1)/k}z\left((1+z^{-1/k}x)^k-1\right)\\
&=&-\sum_{i\ge 1 }{-k+1\choose i}x^iz^{-i/k}(1+z^{-1/k}x)^{k-1}\\
&=&-\sum_{i\ge 0 }{-k+1\choose i}x^iz^{-i/k}\left((1+z^{-1/k}x)^{k-1}\right)+(1+z^{-1/k}x)^{k-1}\\
&=&-1+(1+z^{-1/k}x)^{k-1}\\
&=&\frac{\partial}{\partial z}z\left((1+z^{-1/k}x)^k-1\right)\\
&=&\frac{\partial}{\partial z}\Delta_k^x (z)x.
\end{eqnarray*}

From \cite{BDM} we have that $\Delta_k^x (z)x^n = (\Delta_k^x (z) x)^n$ for $n\in\Z.$ Thus
\begin{eqnarray*}
&&\frac{1}{k}\sum_{i\ge 1 }{-k+1\choose i}\left(-x^i\frac{\partial}{\partial x}\right)z^{-1-(i-1)/k}\Delta_{k}^x(z)x^n\\
&=&\frac{1}{k}\sum_{i\ge 1 }{-k+1\choose i}\left(-x^i\frac{\partial}{\partial x}\right)z^{-1-(i-1)/k}z^n\left((1+z^{-1/k}x)^k-1\right)^n\\
&=&-n\sum_{i\ge 1 }{-k+1\choose i}x^iz^{-i/k}(1+z^{-1/k}x)^{k-1}z^{n-1}\left((1+z^{-1/k}x)^k-1\right)^{n-1}\\
&=&-n\sum_{i\ge 0 }{-k+1\choose i}x^iz^{-i/k}(1+z^{-1/k}x)^{k-1}z^{n-1}\left((1+z^{-1/k}x)^k-1\right)^{n-1}\\
&+&n(1+z^{-1/k}x)^{k-1}z^{n-1}\left((1+z^{-1/k}x)^k-1\right)^{n-1}\\
&=&-nz^{n-1}\left((1+z^{-1/k}x)^k-1\right)^{n-1}+n(1+z^{-1/k}x)^{k-1}z^{n-1}\left((1+z^{-1/k}x)^k-1\right)^{n-1}\\
&=&nz^{n-1}\left((1+z^{-1/k}x)^k-1\right)^{n-1}\left(-1+(1+z^{-1/k}x)^{k-1}\right)\\
&=&n(\Delta_k^x (z) x)^{n-1}\frac{\partial}{\partial z}\Delta_k^x (z)\\
&=&\frac{\partial}{\partial z}(\Delta_k^x (z) x)^n\\
&=&\frac{\partial}{\partial z}\Delta_k^x (z)x^n.
\end{eqnarray*}
Therefore,  equation (\ref{newequation}) holds.
\end{proof}

As the first main result of this section we have:

\begin{theorem}\label{p4.7}
Let $M,N$ and $W$ be weak $V$-modules and let $\sigma=(1\ 2\ \cdots k)\in \text{Aut}(V^{\otimes k})$.
Let ${\mathcal{Y}}(\cdot,z)$ be an intertwining operator of type $\binom{W}{M\ N}.$
 For $w\in M$, set
\begin{align}
\mathcal{\overline{Y}}(w,z)=\mathcal{Y}(\Delta_{k}(z)w,z^{1/k})\in (\text{Hom}(N,W))\{z\}.
\end{align}
Then
\begin{align}
\overline{\mathcal{Y}}(w,z)\in {\mathcal{H}}(T_{\sigma}(N),T_{\sigma}(W))\   \   \  \mbox{ for }w\in M.
\end{align}
Furthermore, the linear map $f:  M\otimes V^{\otimes (k-1)}\rightarrow
\mathcal{H}(T_{\sigma}(N),T_{\sigma}(W))$,
 defined by
\[
f(w\otimes a)=a_{-1}^{\mathcal{H}}\overline{\mathcal{Y}}(w,z)\   \   \   \mbox{ for }w\in M,\ a\in V^{\otimes (k-1)},
\]
 is a weak $V^{\otimes k}$-module homomorphism from $M\otimes V^{\otimes (k-1)}$ to
$\mathcal{H}(T_{\sigma}(N),T_{\sigma}(W)).$
\end{theorem}

\begin{proof} First, we show that $\overline{\mathcal{\mathcal{Y}}}(w,z)\in\mathcal{H}(T_{\sigma}(N),T_{\sigma}(W))$
for $w\in M$.  It is clear that $\overline{\mathcal{\mathcal{Y}}}(w,z)$ satisfies the condition (G1).  For condition (G2),
note that the following relations hold  for any $w\in M,\ m\in \Z$:
\begin{align*}
&\mathcal{Y}(L(-1)w,z)=\frac{d}{dz}\mathcal{Y}(w,z),\\
&[L(m),\mathcal{Y}(w,z)]=\sum_{i\ge 0}{m+1\choose i}z^{m+1-i}\mathcal{Y}(L(i-1)w,z).
\end{align*}
Using these relations and Lemma \ref{newlemma} we have
\begin{eqnarray*}
&&[L_{T_{\sigma}}(-1),\overline{\mathcal{\mathcal{Y}}}(w,z)]\\
&=&\frac{1}{k}\left[L(-k),\mathcal{Y}(\Delta_{k}(z)w,z^{1/k})\right]\\
&=&\frac{1}{k}\sum_{i\ge 0}{-k+1\choose i}(z^{1/k})^{1-k-i}\mathcal{Y}(L(i-1)\Delta_{k}(z)w,z^{1/k})\\
&=&\frac{1}{k}z^{\frac{1}{k}-1}\mathcal{Y}(L(-1)\Delta_{k}(z)w,z^{1/k})
+
\frac{1}{k}\sum_{i\ge 1 }{-k+1\choose i}z^{-1-(i-1)/k}\mathcal{Y}(L(i-1)\Delta_{k}(z)w,z^{1/k})\\
&=&\frac{1}{k}z^{\frac{1}{k}-1}\left(\frac{\partial}{\partial y^{1/k}}\mathcal{Y}(\Delta_{k}(z)w,y^{1/k})\right)|_{y=z}\\
&&+
\frac{1}{k}\sum_{i\ge 1 }{-k+1\choose i}z^{-1-(i-1)/k}\mathcal{Y}(L(i-1)\Delta_{k}(z)w,z^{1/k})\\
&=&\left(\frac{\partial}{\partial y}\mathcal{Y}(\Delta_{k}(z)w,y^{1/k})\right)|_{y=z}+\mathcal{Y}\left(\frac{\partial}{\partial z}\Delta_{k}(z)w,z^{1/k}\right)\\
&=&\frac{d}{dz}\mathcal{Y}(\Delta_{k}(z)w,z^{1/k})\\
&=&\frac{d}{dz}\overline{\mathcal{\mathcal{Y}}}(w,z).
\end{eqnarray*}

For condition (G3), let $u\in V,\ w\in M$. By Lemma 3.3 in \cite{BDM}, we have
\begin{align}\label{COM relation for inter W}
 & Y_{T_{\sigma}(W)}(u^{1},z_{1})\mathcal{\overline{Y}}(w,z_{2})-\mathcal{\overline{Y}}(w,z_{2})Y_{T_{\sigma}(N)}(u^{1},z_{1})\nonumber \\
 =\ & \text{Res}_{z_{0}}\frac{1}{k}z_{2}^{-1}\delta\left(\frac{\left(z_{1}-z_{0}\right)^{1/k}}{z_{2}^{1/k}}\right)\overline{\mathcal{Y}}\left(Y_{M}(u,z_{0})w,z_{2}\right)\nonumber \\
 =\ & \text{Res}_{z_{0}}\frac{1}{k}\sum_{r=0}^{k-1}z_{2}^{-1}\delta\left(\frac{z_{1}-z_{0}}{z_2}\right)
 \left(\frac{z_{1}-z_{0}}{z_2}\right)^{-r/k}\overline{\mathcal{Y}}\left(Y_{M}(u,z_{0})w,z_{2}\right)\nonumber \\
 =\ & \text{Res}_{z_{0}}\frac{1}{k}\sum_{r=0}^{k-1}z_{1}^{-1}\delta\left(\frac{z_{2}+z_{0}}{z_1}\right)
 \left(\frac{z_{2}+z_{0}}{z_1}\right)^{r/k}\overline{\mathcal{Y}}\left(Y_{M}(u,z_{0})w,z_{2}\right).
\end{align}
It follows that there exists a nonnegative integer $n$ such that
\begin{align}
(z_1-z_2)^{n}Y_{T_{\sigma}(W)}(u^{1},z_{1})\mathcal{\overline{Y}}(w,z_{2})
=(z_1-z_2)^{n}\mathcal{\overline{Y}}(w,z_{2})Y_{T_{\sigma}(N)}(u^{1},z_{1}).
\end{align}
Since $\{ \sigma^{j}(u^1)\ |\  u\in V,\ 1\le j\le k\}$ generates $V^{\otimes k}$ as a vertex algebra,
by Lemma \ref{generating-H} we have $\mathcal{\overline{Y}}(w,z)\in \mathcal{H}(T_{\sigma}(N),T_{\sigma}(W))$.
Note that $V^{\otimes k}$ as a vertex algebra is generated by the subset $\{ u^{j}\ |\  u\in V,\ 1\le j\le k\}$.

Second, we show that $f$ is a homomorphism of weak $V$-modules where ${\mathcal{H}}(T_{\sigma}(N),T_{\sigma}(W))$
is viewed as a $V$-module through the embedding $\pi_1$ of $V$ into $V^{\otimes k}$.
Note that for $1\le i\le k$,
\begin{align*}
Y_{T_{\sigma}(W)}(u^{i},z_{1})=Y_{T_{\sigma}(W)}(\sigma^{i-1}u^{1},z_{1})
=\lim_{z_1^{1/k}\rightarrow \eta_k^{1-i}z_1^{1/k}}Y_{T_{\sigma}(W)}(u^{1},z_{1}).
\end{align*}
Then using (\ref{COM relation for inter W}) (by replacing $z_1^{1/k}$ with $\eta_k^{1-i}z_1^{1/k}$), we get
\begin{align}\label{quasi-commutator}
 & Y_{T_{\sigma}(W)}(u^{i},z_{1})\mathcal{\overline{Y}}(w,z_{2})-\mathcal{\overline{Y}}(w,z_{2})Y_{T_{\sigma}(N)}(u^{i},z_{1})\nonumber \\
=\ & \text{Res}_{z_{0}}\frac{1}{k}\sum_{r=0}^{k-1}z_{1}^{-1}\delta\left(\frac{z_{2}+z_{0}}{z_1}\right)
 \left(\frac{z_{2}+z_{0}}{z_1}\right)^{r/k}\eta_{k}^{r(i-1)}\overline{\mathcal{Y}}\left(Y_{M}(u,z_{0})w,z_{2}\right).
\end{align}

For $u\in V$, write
\begin{align}
u^{i}=u^{i,0}+u^{i,1}+\cdots +u^{i,k-1},
\end{align}
 where $u^{i,r} \in (V^{\otimes k})^r$  for $0\le r\le k-1$.  As
$$Y_{T_{\sigma}(W)}(u^{i,r},z_{1})\in z_1^{-\frac{r}{T}}({\rm End} T_{\sigma}(W))[[z_1,z_1^{-1}]],\   \   \   \
Y_{T_{\sigma}(N)}(u^{i,r},z_{1})\in z_1^{-\frac{r}{T}}({\rm End} T_{\sigma}(N))[[z_1,z_1^{-1}]],$$
from (\ref{quasi-commutator}) we get
\begin{align}
 & Y_{T_{\sigma}(W)}(u^{i,r},z_{1})\mathcal{\overline{Y}}(w,z_{2})
 -\mathcal{\overline{Y}}(w,z_{2})Y_{T_{\sigma}(N)}(u^{i,r},z_{1})\nonumber \\
=\  & \text{Res}_{z_{0}}\frac{1}{k}z_{1}^{-1}\eta_k^{(i-1)r}\left(\frac{z_{2}+z_{0}}{z_{1}}\right)^{r/k}
\delta\left(\frac{z_{2}+z_{0}}{z_{1}}\right)\mathcal{\overline{Y}}(Y_M(u,z_{0})w,z_{2}).
 \label{eigenvector quasi commutator}
\end{align}

As $\mathcal{Y}(\cdot,x)$ is an intertwining operator of type $\binom{W}{M\ N}$, we have
\begin{align}\label{Jacobi-Y-delta}
&z_{0}^{-1}\delta\left(\frac{x_{1}^{1/k}-x^{1/k}}{z_{0}}\right)
Y_W\left(\Delta_{k}(x_1)v,x_{1}^{1/k}\right)\mathcal{Y}\left(\Delta_{k}(x)w,x^{1/k}\right)\nonumber\\
 &\  \   -z_{0}^{-1}\delta\left(\frac{-x^{1/k}+x_{1}^{1/k}}{z_{0}}\right)
 \mathcal{Y}\left(\Delta_{k}(x)w,x^{1/k}\right)Y_N\left(\Delta_{k}(x_1)v,x_{1}^{1/k}\right)\nonumber\\
=\ \  &x^{-1/k}\delta\left(\frac{x_{1}^{1/k}-z_{0}}{x^{1/k}}\right)
\mathcal{Y}\left(Y_M(\Delta_{k}(x_{1})v,z_{0})\Delta_{k}(x)w,x^{1/k}\right).
\end{align}
Let $p$ be a nonnegative integer such that $z_0^{p}Y_M(\Delta_{k}(x_{1})v,z_{0})\Delta_{k}(x)w$
involves only nonnegative (integer) powers of $z_0$. Then we have
\begin{align*}
(x_1-x)^p Y_W\left(\Delta_{k}(x_1)v,x_{1}^{1/k}\right)\mathcal{Y}\left(\Delta_{k}(x)w,x^{1/k}\right)=
(x_1-x)^p \mathcal{Y}\left(\Delta_{k}(x)w,x^{1/k}\right)Y_N\left(\Delta_{k}(x_1)v,x_{1}^{1/k}\right),
\end{align*}
which is
\begin{align*}
(x_1-x)^p Y_{T_{\sigma}(W)}(v^1,x_{1})\overline{\mathcal{Y}}(w,x)=
(x_1-x)^p\overline{\mathcal{Y}}(w,x)Y_{T_{\sigma}(N)}(v^1,x).
\end{align*}
Then
\begin{align}\label{one-way}
z^pY_{\mathcal{H}}(v^1,z)\overline{\mathcal{Y}}(w,x)
=\left((x_1-x)^p Y_{T_{\sigma}(W)}(v^1,x_{1})\overline{\mathcal{Y}}(w,x)\right)|_{x_1^{1/k}=(x+z)^{1/k}}.
\end{align}
On the other hand, from (\ref{Jacobi-Y-delta}) we get
\begin{align*}
&x_1^{-1/k}\delta\left(\frac{x^{1/k}+z_{0}}{x_1^{1/k}}\right)
\left((x_1-x)^pY_{T_{\sigma}(W)}(v^1,x_{1})\overline{\mathcal{Y}}(w^1,x)\right)\nonumber\\
=\ &x_1^{-1/k}\delta\left(\frac{x^{1/k}+z_{0}}{x_1^{1/k}}\right)
(x_1-x)^p\mathcal{Y}\left(Y_M(\Delta_{k}(x_{1})v,z_{0})\Delta_{k}(x)w,x^{1/k}\right).
\end{align*}
Taking the residue with respect to $x_1^{1/k}$, we obtain
\begin{align}\label{middle-step1}
&\left((x_1-x)^pY_{T_{\sigma}(W)}(v^1,x_{1})\overline{\mathcal{Y}}(w,x)\right)|_{x_1^{1/k}=x^{1/k}+z_0}\nonumber\\
=\ &
\left((x_1-x)^p\mathcal{Y}\left(Y_M(\Delta_{k}(x_{1})v,z_{0})\Delta_{k}(x)w,x^{1/k}\right)\right)|_{x_1^{1/k}=x^{1/k}+z_0}.
\end{align}
Notice that for any
$$f(x,x_1,z_0)\in x^{\alpha}U((x^{1/T}))((x_1^{1/k},z_0))$$
with $\alpha$ a complex number, $U$ a vector space and $T$ a positive integer, we have
$$\left(f(x,x_1,z_0)|_{x_1^{1/k}=x^{1/k}+z_0}\right)|_{z_0=(x+z)^{1/k}-x^{1/k}}
=f(x,x_1,(x+z)^{1/k}-x^{1/k})|_{x_1^{1/k}=(x+z)^{1/k}}.$$
In view of this, applying substitution $z_0=(x+z)^{1/k}-x^{1/k}$ to (\ref{middle-step1}), we get
\begin{align}
&\left((x_1-x)^pY_{T_{\sigma}(W)}(v^1,x_{1})\overline{\mathcal{Y}}(w,x)\right)|_{x_1^{1/k}=(x+z)^{1/k}}\nonumber\\
=\ &
\left((x_1-x)^p\mathcal{Y}\left(Y_M(\Delta_{k}(x_{1})v,(x+z)^{1/k}-x^{1/k})\Delta_{k}(x)w,x^{1/k}\right)\right)|_{x_1^{1/k}=(x+z)^{1/k}}
\nonumber\\
=\ &z^p\mathcal{Y}\left(Y_M(\Delta_{k}(x+z)v,(x+z)^{1/k}-x^{1/k})\Delta_{k}(x)w,x^{1/k}\right).
\end{align}
Combining this with (\ref{one-way}), and then using (\ref{delta-conjugation}) we get
\begin{align}
Y_{\mathcal{H}}(v^1,z)\overline{\mathcal{Y}}(w,x)
=\ &\mathcal{Y}\left(Y_M(\Delta_{k}(x+z)v,(x+z)^{1/k}-x^{1/k})\Delta_{k}(x)w,x^{1/k}\right)\nonumber\\
=\ &\mathcal{Y}\left(\Delta_{k}(x)Y_{M}(v,z)w,x^{1/k}\right)\nonumber\\
=\ &\overline{\mathcal{Y}}(Y_{M}(v,z)w,x).
\end{align}
Furthermore, for any $v\in V,\ a\in \C{\bf 1}\otimes V^{\otimes (k-1)}$, we have
$$f(Y(v,z)w\otimes a)=a_{-1}^{\mathcal{H}}\overline{\mathcal{Y}}(Y_{M}(v,z)w,x)=a_{-1}^{\mathcal{H}}Y_{\mathcal{H}}(v^1,z)\overline{\mathcal{Y}}(w,x)=Y_{\mathcal{H}}(v^1,z)a_{-1}^{\mathcal{H}}\overline{\mathcal{Y}}(w,x)
=Y_{\mathcal{H}}(v^1,z)f(w\otimes a),$$
noticing that $[a_{m}^{\mathcal{H}},(v^1)_{n}^{\mathcal{H}}]=0$ for $m,n\in \Z$ as $(v^1)_{i}a=0$ for all $i\ge 0$.
This proves that $f$ is a homomorphism of weak $V$-modules.

Third, we show that $f$ is a homomorphism of weak $\C {\bf 1}\otimes V^{\otimes (k-1)}$-modules.
To this end, we first show that if $2\le i\le k$, then
\begin{align}\label{vacuum-like-Vi}
(u^{i})_{n}^{\mathcal{H}}\overline{\mathcal{Y}}(w,x)=0\   \   \   \mbox{ for all }u\in V,\  w\in M,\ n\ge 0.
\end{align}
Recall $$u^i=u^{i,0}+u^{i,1}+\cdots +u^{i,k-1},$$
where $u^{i,r}\in (V^{\otimes k})^{r}$ for $0\le r\le k-1$. By definition, we have
\begin{align*}
 & \left(u^{i,r}\right)_{n}^{\mathcal{H}}\overline{\mathcal{Y}}(w,x)\\
 & =\sum_{j\ge 0}\binom{-\frac{r}{k}}{j}\text{Res}_{x_{1}}x^{-\frac{r}{k}-j}x_1^{\frac{r}{T}}\left\{ (x_{1}-x)^{n+j}Y_{T_{\sigma}(W)}\left(u^{i,r},x_{1}\right)\overline{\mathcal{Y}}(w,x)-(-x+x_{1})^{n+j}\overline{\mathcal{Y}}(w,x)Y_{T_{\sigma}(N)}\left(u^{i,r},x_{1}\right)\right\}
\end{align*}
for $n\in \Z$. Assume $n\ge 0$, using (\ref{eigenvector quasi commutator}) we get
\begin{align*}
 & \left(u^{i}\right)_{n}^{\mathcal{H}}\overline{\mathcal{Y}}(w,x)\nonumber \\
 =\ & \sum_{r=0}^{k-1}\left(u^{i,r}\right)_{n}^{\mathcal{H}}\overline{\mathcal{Y}}(w,x)\nonumber\\
 =\ & \sum_{r=0}^{k-1}\text{Res}_{x_{1}}\sum_{j\ge 0}\binom{-\frac{r}{k}}{j}x^{-\frac{r}{k}-j}x_1^{\frac{r}{k}}
(x_{1}-x)^{n+j}\left(Y_{T_{\sigma}(W)}(u^{i,r},x_{1})\overline{\mathcal{Y}}(w,x)-\overline{\mathcal{Y}}(w,x)Y_{T_{\sigma}(N)}\left(u^{i,r},x_{1}\right)\right)\nonumber \\
 =\ & \text{Res}_{x_{0}}\text{Res}_{x_{1}}\sum_{r=0}^{k-1}\sum_{j\ge 0}\binom{-\frac{r}{k}}{j}x^{-\frac{r}{k}-j}x_1^{\frac{r}{k}}
 (x_{1}-x)^{n+j}\nonumber \\
 &\  \  \   \  \   \
  \cdot\frac{1}{k}\eta_k^{(1-i)r}x_1^{-1}\delta\left(\frac{x_{2}+x_{0}}{x_1}\right)\left(\frac{x_{2}+x_{0}}{x_1}\right)^{r/k}
  \overline{\mathcal{Y}}(Y_M(u,x_{0})w,x)\nonumber\\
  =\ & \text{Res}_{x_{0}}\text{Res}_{x_{1}}\sum_{r=0}^{k-1}\sum_{j\ge 0}\binom{-\frac{r}{k}}{j}x^{-\frac{r}{k}-j}x_1^{\frac{r}{k}}
 x_{0}^{n+j}\nonumber \\
 &\  \  \   \  \   \
  \cdot\frac{1}{k}\eta_{k}^{(i-1)r}x_1^{-1}\delta\left(\frac{x+x_{0}}{x_1}\right)\left(\frac{x+x_{0}}{x_1}\right)^{r/k}
  \overline{\mathcal{Y}}(Y_M(u,x_{0})w,x)\nonumber\\
 =\ & \text{Res}_{x_{0}}\text{Res}_{x_{1}}\sum_{r=0}^{k-1}\left(\frac{x+x_0}{x_1}\right)^{-r/k}
 x_{0}^{n}
  \cdot\frac{1}{k}\eta_k^{(i-1)r}x_1^{-1}\delta\left(\frac{x+x_{0}}{x_1}\right)\left(\frac{x+x_{0}}{x_1}\right)^{r/k}
 \overline{ \mathcal{Y}}(Y(u,x_{0})w,x)\nonumber\\
 =\ & \left(\frac{1}{k}\sum_{r=0}^{k-1}\eta_{k}^{(i-1)r}\right)\text{Res}_{x_{0}}x_0^n\overline{\mathcal{Y}}(Y_M(u,x_{0})w,x)
 \nonumber\\
 =\ &\delta_{i,1}\overline{\mathcal{Y}}(u_nw,x).\label{eq: vacuum-like vector for V^i}
\end{align*}
This proves (\ref{vacuum-like-Vi}).
Since $V^{i}$ for $2\le i\le k$ generate $\C {\bf 1}\otimes V^{\otimes (k-1)}$ as a vertex algebra,
from \cite{LL}, $\overline{\mathcal{Y}}(w,x)$ for $w\in M$ are vacuum-like vectors in
${\mathcal{H}}(T_{\sigma}(N),T_{\sigma}(W))$
viewed as a weak module for  $\C {\bf 1}\otimes V^{\otimes (k-1)}$ in the sense of \cite{Li-form}.
Then by (\cite{LL}; Proposition 4.7.7, or Corollary 4.7.8),
the linear map $f$ from $M\otimes V^{\otimes (k-1)}$ to ${\mathcal{H}}(T_{\sigma}(N),T_{\sigma}(N))$, defined by
$f(w\otimes a)=a_{-1}^{\mathcal{H}}\overline{\mathcal{Y}}(w,x)$ for $w\in M,\ a\in V^{\otimes (k-1)}$, is
a weak module homomorphism for
$\C {\bf 1}\otimes V^{\otimes (k-1)}$. As $V^1$ and $\C {\bf 1}\otimes V^{\otimes (k-1)}$ generate $V^{\otimes k}$ as a vertex algebra, it follows immediately that $f$ is a weak $V^{\otimes k}$-module homomorphism.
This completes the proof.
\end{proof}

Given $V$-modules $M,N$, and $W$, we have the space $I_{V}\binom{W}{M\ N}$ of  intertwining operators.
Furthermore,  we have $\sigma$-twisted $V^{\otimes k}$-modules  $T_{\sigma}(N)$ and $T_{\sigma}(W)$, and
we have the space $ I_{V^{\otimes k}}\binom{T_{\sigma}(W)}{M^{1}\ T_{\sigma}(N)}$
of intertwining operators of type $\binom{T_{\sigma}(W)}{M^{1}\ T_{\sigma}(N)}$.
On the other hand, in view of Theorem \ref{H is a weak module}, we have a weak $V^{\otimes k}$-module
$\left(\mathcal{H}(T_{\sigma}(N),T_{\sigma}(W)),Y_{\mathcal{H}}(\cdot,z)\right)$.
Combining Theorem \ref{p4.7} and Proposition \ref{isomorphism of V-hom and Intertwining}, we immediately have:

\begin{corollary}\label{untwisted-twisted}
For every intertwining operator $\mathcal{Y}(\cdot,z)$ of type $ \binom{W}{M\ N}$,
there exists an intertwining operator $\overline{\mathcal{Y}}(\cdot,z)$
of type $\binom{T_{\sigma}(W)}{M^{1}\ T_{\sigma}(N)}$, which is uniquely determined by
\begin{align}
\overline{\mathcal{Y}}(w^1,z)=\mathcal{Y}(\Delta_k(z)w,z^{1/k})\   \   \   \mbox{ for }w\in M.
\end{align}
\end{corollary}

In view of Corollary \ref{untwisted-twisted}, we have a linear map
\begin{align}\label{def-pi}
\pi: \  \  I_{V}\binom{W}{M\ N}\rightarrow  I_{V^{\otimes k}}\binom{T_{\sigma}(W)}{M^{1}\ T_{\sigma}(N)};
\  \ \mathcal{Y}(\cdot,z)\mapsto \overline{\mathcal{Y}}(\cdot,z).
\end{align}

\subsubsection{From $I_{V^{\otimes k}}\binom{T_{\sigma}(W)}{M^{1}\ T_{\sigma}(N)}$ to
$I_{V}\binom{W}{M\ N}$}

In this section, we shall prove that the linear map $\pi$ defined above from $I_{V}\binom{W}{M\ N}$ to
$I_{V^{\otimes k}}\binom{T_{\sigma}(W)}{M^{1}\ T_{\sigma}(N)}$ is an isomorphism by constructing its inverse map.
The arguments here are similar to those  in \cite{BDM} (Section 4).

Let $V,\ k$, and $\sigma$ be given as before. On the other hand, let $M, N$, and $W$ be $V$-modules.
Set $M^{1}=M\otimes V^{\otimes (k-1)}$, which is a $V^{\otimes k}$-module. Recall
the $\sigma$-twisted $V^{\otimes k}$-modules $T_{\sigma}(N)$ and $T_{\sigma}(W)$.

Recall $\Delta_k(z)$ from (\ref{delta-kx}). Here, we shall need its inverse
\begin{align}
\Delta_{k}(z)^{-1}=z^{(1-1/k)L(0)}k^{L(0)}\exp\left(-\sum_{j\ge 1}a_{j}z^{-j/k}L(j)\right).
\end{align}
As before, by convention we define $k^{\alpha}=e^{\alpha\ln k}$ and $(z^k)^{\alpha}=z^{k\alpha}$ for any $\alpha\in \C$.
Furthermore, for any formal series
$$A(x)=\sum_{\alpha\in \C}u_{\alpha}x^{\alpha}\in U\{x\}$$
with $U$ a vector space, we define
$$A(z^k)=\sum_{\alpha\in \C}u_{\alpha}z^{k\alpha}\in U\{z\}.$$
In particular, if $w$ is an $L(0)$-eigenvector in a $V$-module
 with eigenvalue $\alpha\in \C$, we have
$$z^{(1-1/k)L(0)}k^{L(0)}w=z^{(1-1/k)\alpha}k^{\alpha}w.$$
Notice that $\Delta_{k}(z^k)^{-1}{\bf 1}={\bf 1}$.

The following lemma follows from the proof of Lemma 4.1 in \cite{BDM}:

\begin{lemma}\label{L(-1) property}
Let $\overline{\mathcal{Y}}(\cdot,z)$ be an intertwining
operator of type $\binom{T_{\sigma}(W)}{M^{1}\ T_{\sigma}(N)}$.
For $w\in M$, set
$$\mathcal{Y}(w,z)=\overline{\mathcal{Y}}\left(\left(\Delta_{k}(z^{k})^{-1}w\right)^{1},z^{k}\right),$$
an element of $(\text{Hom}(N,W))\{ z\}$, where $\left(\Delta_{k}(z^{k})^{-1}w\right)^{1}=\Delta_{k}(z^{k})^{-1}w
\otimes {\bf 1}^{\otimes (k-1)}\in M^1\{z\}$. Then
\[
\mathcal{Y}(L(-1)w,z)=\frac{d}{dz}\mathcal{Y}(w,z)
\]
on $N$ for any $w\in M$. 
\end{lemma}

Furthermore, we have the following commutator formula:

\begin{lemma}\label{Commutator formula}
Let $\overline{\mathcal{Y}}(\cdot,z)$ be an intertwining operator of type
$\binom{T_{\sigma}(W)}{M^{1}\ T_{\sigma}(N)}$. Then
\[
Y_{W}(u, z_{1})\mathcal{Y}(w, z_{2})-\mathcal{Y}(w,z_{2})Y_{N}(u,z_{1})
=\text{Res}_{z_{0}}z_{2}^{-1}\delta\left(\frac{z_{1}-z_{0}}{z_{2}}\right)\mathcal{Y}(Y_{M}(u,z_{0})w,z_{2})
\]
for $u\in V$, $w\in M$, where $\mathcal{Y}(w,z)$ for $w\in M$ is defined as in Lemma \ref{L(-1) property}.
\end{lemma}

\begin{proof} Let $u\in V,\ w\in M$. Noticing that
$$(\sigma^{j}u^1)_{n}w^1=(u^{j+1})_{n}w^1=0\   \   \   \mbox{ for all }1\le j\le k-1,\ n\ge 0, $$
from the twisted Jacobi identity (\ref{Twisted Intertwining}), we get
\begin{align*}
 & Y_{T_{\sigma}(W)}(u^{1},z_{1})\mathcal{\overline{Y}}(w^{1},z_{2})
 -\mathcal{\overline{Y}}(w^{1},z_{2})Y_{T_{\sigma}(N)}(u^{1},z_{1})\\
=\ &\text{Res}_{z_{0}}\frac{1}{k}z_{2}^{-1}\delta\left(\frac{\left(z_{1}-z_{0}\right)^{1/k}}{z_{2}^{1/k}}\right)\overline{\mathcal{Y}}\left(Y(u^{1},z_{0})w^{1},z_{2}\right).
\end{align*}
Then we have (see \cite{BDM}, Section 4)
\begin{alignat*}{1}
 & Y_{W}(u, z_{1})\mathcal{Y}(w,z_{2})-\mathcal{Y}(w, z_{2})Y_{N}(u,z_{1})\\
 =\ & Y_{T_{\sigma}(W)}\left(\left(\Delta_{k}\left(z_{1}^{k}\right)^{-1}u\right)^{1},z_{1}^{k}\right)\overline{\mathcal{Y}}\left(\left(\Delta_{k}\left(z_{2}^{k}\right)^{-1}w\right)^{1},z_{2}^{k}\right)\\
 & -\overline{\mathcal{Y}}\left(\left(\Delta_{k}\left(z_{2}^{k}\right)^{-1}w\right)^{1},z_{2}^{k}\right)
 Y_{T_{\sigma}(N)}\left(\left(\Delta_{k}(z_{1}^{k})^{-1}u\right)^{1},z_{1}^{k}\right)\\
=\ & \text{Res}_{x}\frac{1}{k}z_{2}^{-k}\delta\left(\frac{\left(z_{1}^{k}-x\right)^{1/k}}{z_{2}}\right)\overline{\mathcal{Y}}\left(Y\left(\left(\Delta_{k}\left(z_{1}^{k}\right)^{-1}u\right)^{1},x\right)\left(\Delta_{k}(z_{2}^{k})^{-1}w\right)^{1},z_{2}^{k}\right).
\end{alignat*}
We are going to use the change-of-variable $x=z_{1}^{k}-(z_{1}-z_{0})^{k}$.
Note that for any $n\in\mathbb{Z}$,
\begin{align}
\left(z_{1}^{k}-x\right)^{n/k}|_{x=z_{1}^{k}-\left(z_{1}-z_{0}\right)^{k}}=(z_{1}-z_{0})^{n},
\end{align}
where by convention $(z_{1}^{k}-x)^{n/k}=\sum_{j\ge 0}\binom{n/k}{j}(-1)^jz_1^{n-kj}x^j$.
Using (3.12) in \cite{BDM}, we have
\begin{alignat*}{1}
 & Y_{W}(u,\ z_{1})\mathcal{Y}(w,\ z_{2})-\mathcal{Y}(w,\ z_{2})Y_{N}(u,\ z_{1})\\
 =\ &\text{Res}_{z_{0}}z_{2}^{-k}\left(z_{1}-z_{0}\right)^{k-1}\delta\left(\frac{z_{1}-z_{0}}{z_{2}}\right)\overline{\mathcal{Y}}\left(Y\left(\left(\Delta_{k}(z_{1}^{k})^{-1}u\right)^{1},z_{1}^{k}-(z_{1}-z_{0})^{k}\right)\left(\Delta_{k}\left(z_{2}^{k}\right)^{-1}w\right)^{1},z_{2}^{k}\right)\\
=\  & \text{Res}_{z_{0}}z_{2}^{-1}\delta\left(\frac{z_{1}-z_{0}}{z_{2}}\right)\overline{\mathcal{Y}}\left(Y\left(\left(\Delta_{k}(z_{1}^{k})^{-1}u\right)^{1},z_{1}^{k}-(z_{1}-z_{0})^{k}\right)\left(\Delta_{k}(z_{2}^{k})^{-1}w\right)^{1},z_{2}^{k}\right)\\
=\ & \text{Res}_{z_{0}}z_{2}^{-1}\delta\left(\frac{z_{1}-z_{0}}{z_{2}}\right)\overline{\mathcal{Y}}\left(\left(Y\left(\Delta_{k}\left(\left(z_{2}+z_{0}\right)^{k}\right)^{-1}u,(z_{2}+z_{0})^{k}-z_{2}^{k}\right)\Delta_{k}(z_{2}^{k})^{-1}w\right)^{1},z_{2}^{k}\right).
\end{alignat*}
Then it suffices to prove
\begin{equation}\label{Delta property}
Y\left(\Delta_{k}\left((z_{2}+z_{0})^{k}\right)^{-1}u,(z_{2}+z_{0})^{k}-z_{2}^{k}\right)\Delta_{k}(z_{2}^{k})^{-1}
=\Delta_{k}(z_{2}^{k})^{-1}Y(u,z_{0}),
\end{equation}
or equivalently
\[
\Delta_{k}(z_{2}^{k})Y\left(\Delta_{k}\left((z_{2}+z_{0})^{k}\right)^{-1}u,(z_{2}+z_{0})^{k}-z_{2}^{k}\right)\Delta_{k}(z_{2}^{k})^{-1}=Y(u,z_{0}).
\]
 This last relation follows from (4.24) (\cite{BDM}; Proposition 2.2)
by replacing $u,z$ and $z_{0}$ with $\Delta_{k}((z_{2}+z_{0})^{k})^{-1}u$,
$z_{2}^{k}$ and $(z_{2}+z_{0})^{k}-z_{2},$ respectively.
Now, the proof is complete.
\end{proof}

As the second main result of this section we have:

\begin{theorem}\label{t5.3}
Let $M$, $N$ and $W$ be $V$-modules and let $\sigma=(1\ 2\ \cdots k)\in \text{Aut }(V^{\otimes k}).$
Set $M^{1}=M\otimes V^{\otimes (k-1)}$.
Let $\overline{\mathcal{Y}}(\cdot,z)$ be any intertwining operator of type
$\binom{T_{\sigma}(W)}{M^{1}\ T_{\sigma}(N)}$. For $w\in M$, set
$$\mathcal{Y}(w,z)=\overline{\mathcal{Y}}\left(\left(\Delta_{k}(z^{k})^{-1}w\right)^{1},z^{k}\right),$$
an element of $(\text{Hom}(N,W))\{ z\} $. Then $\mathcal{Y}(\cdot,z)$
is an intertwining operator of type $ \binom{W}{M\ N}.$
\end{theorem}

\begin{proof} As the $L(-1)$-property has been given
in Lemma \ref{L(-1) property}, it remains to prove the Jacobi identity
which is equivalent to the commutator formula which was obtained in Lemma \ref{Commutator formula}
and the weak associativity which states that for any $u\in V,\  w\in M,$
$a\in N,$ there exists a nonnegative integer $n$ such that
\[
(z_{0}+z_{2})^{n}Y_{W}(u,z_{0}+z_{2})\mathcal{Y}(w,z_{2})a
=(z_{2}+z_{0})^{n}\mathcal{Y}(Y_{M}(u,z_{0})w, z_{2})a.
\]
Let $u\in V,\ w\in M$. Recall
$$u^{1}=u\otimes {\bf 1}^{\otimes (k-1)}\in V^{\otimes k},\   \   \  \
w^{1}=w\otimes {\bf 1}^{\otimes (k-1)}\in M^{1}.$$
We have $u^1=\sum_{r=0}^{k-1}u^{1,r}$, where $u^{i,r}\in (V^{\otimes k})^r$, i.e.,
$\sigma u^{1,r}=\eta_k^{r}u^{1,r}$.
By the twisted Jacobi identity, there exists a positive integer
$l$ such that for $n\ge l$,
\[
(x_{0}+x_{2})^{r/k+n}Y_{T_{\sigma}(W)}(u^{1,r},x_{0}+x_{2})\mathcal{\overline{Y}}(w^{1},x_{2})a
=(x_{2}+x_{0})^{r/k+n}\overline{\mathcal{Y}}\left(Y_{M^{1}}(u^{1,r},x_{0})w^{1},x_{2}\right)a
\]
for $r=0,1,\dots,k-1$. Using change-of-variables $x_{2}=z_{2}^{k}$ and $x_{0}=(z_{0}+z_{2})^{k}-z_{2}^{k}$, we obtain
\begin{gather*}
(z_{0}+z_{2})^{r+kn}Y_{T_{\sigma}(W)}\left(u^{1,r},(z_{0}+z_{2})^{k}\right)\mathcal{\overline{Y}}\left(w^{1},z_{2}^{k}\right)a
=(z_{2}+z_{0})^{r+kn}\mathcal{\overline{Y}}\left(Y_{M^{1}}\left(u^{1,r},(z_{0}+z_{2})^{k}-z_{2}^{k}\right)w^{1},z_{2}^{k}\right)a.
\end{gather*}
Thus there exists a positive integer $p$ such that
\begin{gather*}
(z_{0}+z_{2})^{n}Y_{T_{\sigma}(W)}\left(u^{1,r},(z_{0}+z_{2})^{k}\right)\mathcal{\overline{Y}}\left(w^{1},z_{2}^{k}\right)a
=(z_{2}+z_{0})^{n}\mathcal{\overline{Y}}\left(Y_{M^{1}}\left(u^{1,r},(z_{0}+z_{2})^{k}-z_{2}^{k}\right)w^{1},z_{2}^{k}\right)a
\end{gather*}
for $n\ge p,\ r=0,1,\dots,k-1$. Consequently, we have
\begin{gather*}
(z_{0}+z_{2})^{n}Y_{T_{\sigma}(W)}\left(u^{1},(z_{0}+z_{2})^{k}\right)\mathcal{\overline{Y}}\left(w^{1},z_{2}^{k}\right)a=(z_{2}+z_{0})^{n}\mathcal{\overline{Y}}\left(Y_{M^{1}}\left(u^{1},(z_{0}+z_{2})^{k}-z_{2}^{k}\right)w^{1},z_{2}^{k}\right)a
\end{gather*}
for $n\ge p$.

Notice that $\Delta_k(z^k)^{-1}u\in V[z,z^{-1}]$ and $\Delta_{k}(z^{k})^{-1}w\in M[z,z^{-1}]$.
(Both of them are finite sums.)
Then,  there exists a positive integer $q$ such that
\begin{align*}
&(z_{0}+z_{2})^{n}Y_{T_{\sigma}(W)}\left((\Delta_k(z_0+z_2)^{k})^{-1}u)^{1},(z_{0}+z_{2})^{k}\right)\mathcal{\overline{Y}}\left((\Delta_k(z_2^k)^{-1}w)^{1},z_{2}^{k}\right)a\\
=\ &(z_{2}+z_{0})^{n}\mathcal{\overline{Y}}\left(Y_{M^{1}}\left((\Delta_k(z_0+z_2)^{k})^{-1}u)^{1},(z_{0}+z_{2})^{k}-z_{2}^{k}\right)(\Delta_k(z_2^k)^{-1}w)^{1},z_{2}^{k}\right)a
\end{align*}
for $n\ge q$.
Therefore,  we obtain
\begin{align*}
 & (z_{0}+z_{2})^{n}Y_{W}(u,z_{0}+z_{2})\mathcal{\mathcal{Y}}(w,z_{2})a\\
 =\ & (z_{0}+z_{2})^{n}Y_{T_{\sigma}(W)}\left(\left(\Delta_{k}\left(\left(z_{0}+z_{2}\right)^{k}\right)^{-1}u\right)^{1},(z_{0}+z_{2})^{k}\right)\mathcal{\overline{Y}}\left(\left(\Delta_{k}\left(z_{2}^{k}\right)^{-1}w\right)^{1},z_{2}^{k}\right)a\\
=\ & \left(z_{0}+z_{2}\right)^{n}\overline{\mathcal{Y}}\left(Y_{M^{1}}\left(\left(\Delta_{k}\left((z_{2}+z_{0})^{k}\right)^{-1}u\right)^{1},(z_{2}+z_{0})^{k}-z_{2}^{k}\right)\left(\Delta\left(z_{2}^{k}\right)^{-1}w\right)^{1},z_{2}^{k}\right)a\\
=\ & (z_{0}+z_{2})^{n}\overline{\mathcal{Y}}\left(\left(\Delta_{k}\left(z_{2}^{k}\right)^{-1}Y(u,z_{0})w\right)^{1},z_{2}^{k}\right)a\\
=\ & (z_{0}+z_{2})^{n}\mathcal{Y}\left(Y_{M}(u,z_{0})w,z_{2}\right)a
\end{align*}
for $n\ge p$, where we are using (\ref{Delta property}) for the second last equality. This completes the proof.
\end{proof}

To summarize, we have:

\begin{corollary}\label{summary}
The linear map $\pi$ defined in (\ref{def-pi}) is a linear isomorphism
from  $I_{V}\binom{W}{M\ N}$ to $I_{V^{\otimes k}}\binom{T_{\sigma}(W)}{M^{1}\ T_{\sigma}(N)}$.
\end{corollary}

\section{Tensor products of $V^{\otimes k}$-modules with permutation automorphism twisted $V^{\otimes k}$-modules}

In this section, we study tensor product using the results of Section 5
from the $k$-cycle $\sigma$ to an arbitrary permutation in $S_{k}$ and present the second main theorem
of this paper. Some of the results in this section are valid for arbitrary vertex operator algebras.

Let $V$ be a general vertex operator algebra for now. Recall Definition \ref{d2.13} for a tensor product
of a $g_1$-twisted $V$-module $M$ with a $g_2$-twisted $V$-module $N$.
As a convention,  we shall use $M\boxtimes N$ to denote a generic tensor product module,
provided that its existence is justified.

\begin{remark}\label{tensor-product-module-VOA}
Let $V$ be any vertex operator algebra. Note that for any $V$-module $(M,Y_{M})$,
 the pair $(M,Y_{M})$ is a tensor product of $V$ and $M$.
Furthermore, $(M,Y_{M}^{o})$ is a tensor product of $M$ and $V$, where $Y_M^o(\cdot,z)$ is defined by
$$Y_{M}^{o}(w,z)v=e^{zL(-1)}Y_{M}(v,-z)w\   \   \  \mbox{ for }v\in V,\ w\in M$$
(see \cite{FHL}). Thus $M\boxtimes V\simeq M\simeq V\boxtimes M$.
\end{remark}

\begin{remark}\label{tensor-product-module-tensorVOA}
Let $U$ and $V$ be vertex operator algebras and
let $(W,Y_{W})$ be a $U$-module and $(M,Y_M)$ a $V$-module.
Then it is straightforward to show that $W\otimes M$ is a tensor product module for $U\otimes V$-modules $W\otimes V$
and $U\otimes M$, where
\begin{align}
\mathcal{Y}(w\otimes v,z)(u\otimes m)=e^{zL(-1)}Y_{W}(v,-z)w\otimes Y_M(u,z)m
\end{align}
for $w\in W,\ v\in V,\ u\in U,\ m\in M$. That is, $W\otimes M\simeq (W\otimes V)\boxtimes (U\otimes M)$.
\end{remark}

Let $k$ be a positive integer. Recall from \cite{BDM} that for any $V$-module $W$,
we have a $\sigma$-twisted $V^{\otimes k}$-module
$T_{\sigma}(W)$ which equals $W$ as a vector space and
a $V$-module homomorphism from $W_1$ to $W_2$ is exactly the same as
 a $\sigma$-twisted $V^{\otimes k}$-module homomorphism from $T_{\sigma}(W_1)$ to $T_{\sigma}(W_2)$.

Let $M$ and $N$ be $V$-modules. Assume that
$(M\boxtimes N, \mathcal{Y})$ is a tensor product of $M$ and $N$. That is,
$M\boxtimes N$ is a $V$-module and $\mathcal{Y}$ is an intertwining operator of type $\binom{M\boxtimes N}{M\ N}$,
satisfying the universal property.
From Theorem \ref{p4.7}, we have  an intertwining operator $\overline{\mathcal{Y}}$
of type $\binom{T_{\sigma}(M\boxtimes N)}{M^1\ T_{\sigma}(N)}$. It follows from Corollary \ref{summary} that
$(T_{\sigma}(M\boxtimes N),\overline{\mathcal{Y}})$ is a tensor product of $M\otimes V^{\otimes (k-1)}$ with
$T_{\sigma}(N)$. Thus we have proved:

\begin{theorem}\label{t3.6}
Let $M$ and $N$ be $V$-modules. Suppose that there exists a tensor product $(M\boxtimes N, \mathcal{Y})$ of $M$ and $N$.
Then a tensor product of $M\otimes V^{\otimes (k-1)}$ with $T_{\sigma}(N)$ exists and
\[
(M\otimes V^{\otimes (k-1)})\boxtimes T_{\sigma}(N)\simeq T_{\sigma}(M\boxtimes N)
\]
In particular, we have $(M\otimes V^{\otimes(k-1)})\boxtimes T_{\sigma}(V)\simeq T_{\sigma}(M)$ and
Theorem A (in Introduction) is true.
\end{theorem}

Next, we shall generalize this result.
First, we formulate the following lemma which is straightforward to prove:

\begin{lemma}\label{sigma-tau}
Let $U$ be a vertex operator algebra and let $\sigma,\tau$ be automorphisms of $U$ with $\sigma$ of finite order.
Assume that $W_1,W_2$ are $\sigma$-twisted $U$-modules and
$W$ is a $V$-module on which $\tau$ acts such that
\begin{align}
\tau Y_{W}(u,z)w=Y_{W}(\tau (u),z)\tau w\  \   \  \mbox{ for }u\in U,\ w\in W.
\end{align}
If $I(\cdot,z)$ is an intertwining operator of type $\binom{W_2}{W\ W_1}$, then
$I^{\tau^{-1}}(\cdot,z)$ is an intertwining operator of type $\binom{W_2}{W^{\tau}\ W_1}$,
where $W^\tau=W$ as a vector space, $Y_{W^{\tau}}(u,z)=Y_{W}(\tau(u),z)$ for $u\in U$, and
$$I^{\tau^{-1}}(w,z)=I(\tau^{-1}(w),z)\   \   \   \mbox{ for }w\in W.$$
Furthermore, if $(P,\mathcal{Y})$ is a tensor product of $W$ and $W_1$, then
$(P,\mathcal{Y}^{\tau^{-1}})$ is a tensor product of $W^{\tau}$ and $W_1$. In particular,  we have
$W\boxtimes W_1\simeq W^{\tau}\boxtimes W_1$, provided that one of the two tensor products exists.
\end{lemma}

We now in a position to prove Theorem B (in Introduction).


\begin{proof} First of all, combining Theorem \ref{t3.6} with Lemma \ref{sigma-tau}, we have
\begin{align*}
(V^{\otimes (i-1)}\otimes M\otimes V^{\otimes (k-i)})\boxtimes T_{\sigma}(N)\simeq T_{\sigma}(M\boxtimes N)
\end{align*}
for any $V$-module $M$ and for $1\le i\le k$. This shows that this theorem is true
if  for some $1\le r\le k$, $M_{i}=V$ for all $i\ne r$, as $M_{1}\boxtimes\cdots\boxtimes M_{k}\simeq M_r$.
For the general case, notice that as $M_i\boxtimes V\simeq M_i\simeq V\boxtimes M_i$ for $1\le i\le k$, we have
\[
M_{1}\otimes\cdots\otimes M_{k}\simeq \boxtimes_{i=1}^{k}(V^{\otimes(i-1)}\otimes M_{i}\otimes V^{\otimes(k-i)}),
\]
where we are also using Remark \ref{tensor-product-module-tensorVOA}.
Then the general case follows from  the special case and Theorem \ref{t3.5} (an associativity).
\end{proof}

Now, we consider an arbitrary permutation $\sigma\in S_{k}.$
Recall that a partition of $k$ is a sequence of positive integers $k_1,k_2,\dots,k_s$ such that
$$k_1\ge k_2\ge \cdots \ge k_s\ge 1\   \mbox{ and }\  k_1+\cdots +k_s=k.$$
Associated to each partition $\kappa=(k_1,k_2,\dots,k_s)$, we have a permutation
\begin{align}\label{sigma-kappa}
\sigma_{\kappa}:=\sigma_1\circ \sigma_2\circ \cdots \circ \sigma_s,
\end{align}
where $\sigma_1=(12\cdots k_1),\ \sigma_2=((k_1+1)\cdots (k_1+k_2))$, and so  on.
A classical fact is that every permutation in $S_k$ is conjugate to $\sigma_{\kappa}$
for a uniquely determined partition $\kappa$ of $k$.

\begin{remark}\label{conjugation-auto}
Let $U$ be a vertex operator algebra and let $\sigma, \tau$ be finite order automorphisms of $U$ such that
$\tau=\mu \sigma\mu^{-1}$ for some automorphism $\mu$ of $U$.
It is straightforward to show that a linear map
$Y_W(\cdot,z): U\rightarrow ({\rm End }W)\{z\}$ is a $\tau$-twisted $U$-module structure on $W$ if and only if
$Y_W^{\mu}(\cdot,z)$ is a $\sigma$-twisted $U$-module structure on $W$, where
$$Y_W^{\mu}(u,z)=Y_W(\mu (u),z)\   \   \  \mbox{ for }u\in U.$$
If $W$ is a $\tau$-twisted $U$-module,
we denote by $W^{\tau}$ the corresponding $\sigma$-twisted $U$-module.
On the other hand, if $W_1,W_2$ are $\tau$-twisted $U$-modules, a $\tau$-twisted $U$-module homomorphism from $W_1$ to $W_2$ is exactly the same  as a $\sigma$-twisted $U$-module homomorphism from
$W_1^{\mu}$ to $W_2^{\mu}$.
Therefore,  the category of $\tau$-twisted $U$-modules
is  isomorphic to the category of $\sigma$-twisted $U$-modules canonically.
\end{remark}

Furthermore, by a straightforward argument we have:

\begin{lemma}\label{module-tensor-twisted-module-intertwining-operator}
Let $V$ be any vertex operator algebra and let $\sigma$ and $\mu$ be automorphisms of $V$ with $\sigma$ of finite order.  Set $\tau=\mu\sigma\mu^{-1}$.
Assume that $W$ is a $V$-module and $T_1, T_2$ are $\tau$-twisted $V$-modules. Then an intertwining operator
of type $\binom{T_2}{W\ T_1}$ is exactly the same as an intertwining operator
of type $\binom{T_2^{\mu}}{W^{\mu}\ T_1^{\mu}}$.
\end{lemma}

\begin{proposition}\label{module-tensor-twisted-module-congugation}
Let $V$ be any vertex operator algebra and let $\sigma$ and $\mu$ be automorphisms of $V$ with $\sigma$ of finite order.  Set $\tau=\mu\sigma\mu^{-1}$.
 Let $W$ be a $V$-module and $T$ a $\tau$-twisted $V$-module.
 Assume that there exists a tensor product $W^{\mu^{-1}}\boxtimes T$ (a $\tau$-twisted $V$-module).
Then there exists a tensor product $W\boxtimes T^{\mu}$ (a $\sigma$-twisted $V$-module) and
\begin{align}
W\boxtimes T^{\mu}\simeq (W^{\mu^{-1}}\boxtimes T)^{\mu}.
\end{align}
\end{proposition}

\begin{proof} Recall that $T^{\mu}$ is a $\sigma$-twisted $V$-module and
$W^{\mu^{-1}}$ is a $V$-module with $W$ as the underlying space,
 where the vertex operator map $Y_{W}^{\mu^{-1}}(\cdot,x)$ is given by
$$Y_{W}^{\mu^{-1}}(v,x)=Y_{W}(\mu^{-1}v,x)\   \   \  \  \mbox{ for }v\in V.$$
From definition, we have an intertwining operator $\mathcal{Y}$ of type
$\binom{W^{\mu^{-1}}\boxtimes T}{W^{\mu^{-1}}\ T}$, satisfying the universal property.
Let $P$ be any $\sigma$-twisted $V$-module and $I(\cdot,z)$ be any intertwining operator of type
of $\binom{P}{W\ T^{\mu}}$.
By Lemma \ref{module-tensor-twisted-module-intertwining-operator},
$I(\cdot,z)$ is also an intertwining operator of
type $\binom{P^{\mu^{-1}}}{W^{\mu^{-1}}\ T}$. Then there exists a $\tau$-twisted $V$-module homomorphism
$\psi$ from $W^{\mu^{-1}}\boxtimes T$ to $P^{\mu^{-1}}$ such that
$I(w,z)t=\psi (\mathcal{Y}(w,z)t)$ for $w\in W,\ t\in T$. Note that from Remark \ref{conjugation-auto},
$\psi$ is also a $\sigma$-twisted $V$-module homomorphism
from $(W^{\mu^{-1}}\boxtimes T)^{\mu}$ to $P$. This proves that
$((W^{\mu^{-1}}\boxtimes T)^{\mu},\mathcal{Y})$ is a tensor product of $W$ with $T^{\mu}$. Therefore, we have
$(W^{\mu^{-1}}\boxtimes T)^{\mu}\simeq W\boxtimes T^{\mu}$, as desired.
\end{proof}

\begin{remark}\label{tensor-twisted-module}
Let $V_1$ and $V_2$ be vertex operator algebras and let $\tau_1,\tau_2$ be finite order automorphisms
of $V_1,V_2$, respectively. Set $\tau=\tau_1\otimes \tau_2$, an automorphism of $V_1\otimes V_2$.
Then the same arguments of \cite{FHL} show that if $W_i$ is an irreducible $\tau_i$-twisted $V_i$-module for $i=1,2$,
then $W_1\otimes W_2$ is naturally an irreducible $\tau$-twisted $V_1\otimes V_2$-module
and on the other hand, every irreducible $\tau$-twisted $V_1\otimes V_2$-module is isomorphic to one in this form.
Furthermore, if $V_i$ is $\tau_i$-rational for $i=1,2$, then $V_1\otimes V_2$ is $\tau$-rational
(see \cite{BDM}, Lemma 6.1 and Proposition 6.2).
\end{remark}

As before, we view the symmetric group $S_k$ as an automorphism group of $V^{\otimes k}$.
From  Remarks \ref{conjugation-auto} and \ref{tensor-twisted-module} and Theorem \ref{BDM} (due to \cite{BDM}) we immediately have:

\begin{proposition}
Let $V$ be any vertex operator algebra and $\sigma\in S_k$ be any permutation. Suppose that $\mu \sigma\mu^{-1}=\sigma_{\kappa}$, where
$\mu\in S_k$ and $\kappa$ is the partition of $k$ which is uniquely determined by $\sigma$.
Then for any $s$ irreducible $V$-modules $N_{1},...,N_{s}$,
we have an irreducible $\sigma$-twisted $V^{\otimes k}$-module
\begin{align}
T_{\sigma}(N_1,\dots,N_s):=(T_{\sigma_{1}}(N_{1})\otimes\cdots\otimes T_{\sigma_{s}}(N_{s}))^{\mu},
\end{align}
where $\sigma_1,\dots,\sigma_s$ are disjoint cycles as in (\ref{sigma-kappa}),
and on the other hand, every irreducible $\sigma$-twisted
$V^{\otimes k}$-module is isomorphic to one in this form. Furthermore, for
irreducible $V$-modules $N_{1},...,N_{s}$ and $P_1,\dots,P_s$,
$T_{\sigma}(N_1,\dots,N_s)\simeq T_{\sigma}(P_1,\dots,P_s)$ if and only if $N_i\simeq P_i$ for $1\le i\le s$.
\end{proposition}

To present our general result, we also need the following result:

\begin{proposition}\label{modules-twisted-modules-tensor-product}
Let $V_1,\dots,V_s$ be vertex operator algebras, let $\tau_1,\dots,\tau_s$
be finite order automorphisms  of $V_1, \dots,V_s$, respectively. Assume that $V_i$ is $\tau_i$-rational for $1\le i\le s$. Set
$$\tau=\tau_1\otimes \cdots \otimes \tau_s\in \text{Aut }(V_1\otimes \cdots \otimes V_s).$$
Assume that $W_i$ is a $V_i$-module and $T_i$ is a $\tau_i$-twisted $V_i$-module for $i=1,\dots,s$ such that
$\dim I_{V_i}\binom{P_i}{W_i\ T_i}<\infty$  for every irreducible $\tau_i$-twisted $V_i$-module $P_i$ for $1\le i\le s$.
Then there exist tensor products $W_1\boxtimes T_1,\dots, W_s\boxtimes T_s$,
$(W_1\otimes\cdots \otimes  W_s)\boxtimes (T_1\otimes \cdots \otimes T_s)$, and
\begin{align}
(W_1\boxtimes T_1)\otimes \cdots \otimes (W_s\boxtimes T_s)
\simeq (W_1\otimes\cdots \otimes  W_s)\boxtimes (T_1\otimes \cdots \otimes T_s).
\end{align}
\end{proposition}

\begin{proof}
By induction, it suffices to prove that it is true for $s=2$. From Remark \ref{tautological},
there exist tensor products $W_1\boxtimes T_1$ and $W_2\boxtimes T_2$.
By definition, we have an intertwining operator $\mathcal{Y}_{1}$ of type
$\binom{W_1\boxtimes T_1}{W_1\ T_1}$ and an intertwining operator $\mathcal{Y}_{2}$ of type
$\binom{W_2\boxtimes T_2}{W_2\ T_2}$, satisfying the universal property. Define a linear map
$$\mathcal{Y}(\cdot,z)\cdot:  (W_1\otimes W_2)\otimes (T_1\otimes T_2)
\rightarrow \left((W_1\boxtimes T_1)\otimes (W_2\boxtimes T_2)\right)\{z\}$$
by
\begin{align}
\mathcal{Y}(w_1\otimes w_2,z)(t_1\otimes t_2)=\mathcal{Y}_{1}(w_1,z)t_1\otimes \mathcal{Y}_{2}(w_2,z)t_2
\end{align}
for $w_i\in W_i,\ t_i\in T_i$ with $i=1,2$. Then  $\mathcal{Y}$ is an intertwining operator of type
$\binom{(W_1\boxtimes T_1)\otimes (W_2\boxtimes T_2)}{W_1\otimes W_2\ T_1\otimes T_2}$
(see the comments right before Theorem \ref{ADL-twisted-version}).
We claim that $\left((W_1\boxtimes T_1)\otimes (W_2\boxtimes T_2), \mathcal{Y}\right)$ is a tensor product of
$W_1\otimes W_2$ and $T_1\otimes T_2$.
Let $T$ be any $\tau$-twisted $V_1\otimes V_2$-module and let $I$ be any
intertwining operator of type $\binom{T}{W_1\otimes W_2\ T_1\otimes T_2}$.
We must show that there exists a $\tau$-twisted $V_1\otimes V_2$-module homomorphism
$\psi$ from $(W_1\boxtimes T_1)\otimes (W_2\boxtimes T_2)$ to $T$ such that $I=\psi\circ \mathcal{Y}$.
As $V_1\otimes V_2$ is $\tau$-rational,
it suffices to consider the case with $T$ an irreducible $\tau$-twisted $V_1\otimes V_2$-module.
From Remark \ref{tensor-twisted-module}, $T\simeq P_{1}\otimes P_{2}$, where
$P_{i}$ is an irreducible $\tau_i$-twisted $V_i$-module for $i=1,2$.
As $\dim I_{V_i}\binom{P_i}{W_i\ T_i}<\infty$ by assumption,
from Theorem \ref{ADL-twisted-version} (due to \cite{ADL}),
$I=\sum_{j=1}^{r} I_{1,j}\otimes I_{2,j}$, where $I_{1,j}$ are intertwining operators of type
$\binom{P_{1}}{W_1\ T_1}$ and $I_{2,j}$ are intertwining operators of type
$\binom{P_{2}}{W_2\ T_2}$. By the universal property of $(W_i\boxtimes T_i,\mathcal{Y}_{i})$ for $i=1,2$,
there exist $\tau_i$-twisted $V_i$-module
homomorphisms $\psi_{i,j}: W_i\boxtimes T_i\rightarrow P_{i}$ such that
$I_{i,j}(\cdot,z)=\psi_{i,j}\circ \mathcal{Y}_{i}(\cdot,z)$ for $i=1,2$, $1\le j\le r$. Set
$\psi=\sum_{j=1}^r\psi_{1,j}\otimes \psi_{2,j}$, which is a $\tau$-twisted $V_1\otimes V_2$-module homomorphism
from $(W_1\boxtimes T_1)\otimes (W_2\boxtimes T_2)$ to $T$ such that $I=\psi\circ \mathcal{Y}$. This proves that
$((W_1\boxtimes T_1)\otimes (W_2\boxtimes T_2), \mathcal{Y})$ is a tensor product of $(W_1\otimes W_2)$ with $T_1\otimes T_2$. Thus tensor product
$(W_1 \otimes  W_2)\boxtimes (T_1\otimes T_2)$ exists
and $(W_1\boxtimes T_1)\otimes (W_2\boxtimes T_2)\simeq (W_1\otimes  W_2)\boxtimes (T_1\otimes T_2)$.
\end{proof}

We now prove Theorem C.


\begin{proof}  Using Proposition \ref{modules-twisted-modules-tensor-product} and
Theorem B, we obtain
\begin{align*}
&\left(M_{1}\otimes\cdots\otimes M_{k}\right)\boxtimes T_{\sigma_{\kappa}}(N_{1},N_2,\dots,N_s)\\
=\ &  \left(M_{1}\otimes\cdots\otimes M_{k}\right)\boxtimes
\left(T_{\sigma_1}(N_{1})\otimes \cdots \otimes T_{\sigma_s}(N_s)\right)\\
\simeq \ &\left(M^{[k_1]}\boxtimes T_{\sigma_1}(N_1)\right)\otimes \cdots \otimes
\left(M^{[k_s]}\boxtimes T_{\sigma_s}(N_1)\right)\\
\simeq \ &T_{\sigma_1}\left(M^{[k_1]}\boxtimes N_1\right)\otimes \cdots
\otimes T_{\sigma_s}\left(M^{[k_s]}\boxtimes N_s\right)\\
=\ & T_{\sigma_\kappa}\left( (M^{[k_1]}\boxtimes N_1),\dots, (M^{[k_s]}\boxtimes N_s) \right),
\end{align*}
proving the first assertion.
Then using Proposition \ref{module-tensor-twisted-module-congugation} we get the second assertion.
\end{proof}

\vskip10pt {\footnotesize{}{ }\textbf{\footnotesize{}C. Dong}{\footnotesize{}:
Department of Mathematics, University of California Santa Cruz, CA 95064 USA; }\texttt{dong@ucsc.edu}{\footnotesize \par}

\textbf{\footnotesize{}H. Li}{\footnotesize{}:  Department of Mathematical Sciences, Rutgers University, Camden, NJ 08102 USA; }\texttt{hli@camden.rutgers.edu}{\footnotesize \par}

\textbf{\footnotesize{}F. Xu}{\footnotesize{}:  Department of Mathematics, University of California, Riverside, CA 92521 USA; }\texttt{xufeng@math.ucr.edu}{\footnotesize \par}

\textbf{\footnotesize{}N. Yu}{\footnotesize{}: School of Mathematical
Sciences, Xiamen University, Fujian, 361005, CHINA;} \texttt{
ninayu@xmu.edu.cn}{\footnotesize \par}

\end{document}